\documentclass[11pt,A4]{article}
\usepackage[applemac]{inputenc}
\usepackage[english]{babel}
\usepackage{amsmath,amsfonts,amssymb,amsthm,mathrsfs,bbm}
\usepackage[font=sf, labelfont={sf,bf}, margin=1cm]{caption}
\usepackage{graphicx,graphics}
\usepackage{epsfig}
\usepackage{latexsym}
 \usepackage{ae,aecompl}
\usepackage{pstricks}
\usepackage{enumerate}
\usepackage{xcolor}
\usepackage{pifont}
\usepackage[pdfpagemode=UseNone,bookmarksopen=false,colorlinks=true,urlcolor=blue,citecolor=blue,citebordercolor=blue,linkcolor=blue]{hyperref}
\usepackage[margin=2.5cm]{geometry}
\usepackage{comment}
\usepackage{mathpazo} 
\usepackage{eulervm}
\DeclareGraphicsExtensions{.jpg,.pdf}
\usepackage[normalem]{ulem}
\usepackage{enumerate,stmaryrd}
\usepackage{tikz}						
	\usetikzlibrary{arrows, patterns, calc}		
\usepackage{mathtools}					
\usepackage{microtype,etex}					

\usepackage{enumitem}

\newtheorem{thm}{Theorem}
\newtheorem{lem}[thm]{Lemma}
\newtheorem{prop}[thm]{Proposition}
\newtheorem{cor}[thm]{Corollary}

\theoremstyle{definition}

\newtheorem{ex}[thm]{Example}

\newtheorem{ques}[thm]{Question}

\renewcommand\Pr[1]{\mathbb{P}\left(#1\right)}
\newcommand\Es[1]{\mathbb{E}\left[#1\right]}

\def \P {\mathbb{P}}

\def \Cov {\mathrm{Cov}}

\def \N {\mathbb N}

\def \D {\mathbb D}
\def \NC {\mathbb{NC}}
\def \T {\mathbb T}
\def \z {\zeta}

\def \R {\mathbb R}
\def \Q {\mathbb Q}
\def \D {\mathbb D}
\def \Z {\mathbb Z}
\def \C {\mathbb C}
\def \S {\mathbb{S}^1}

\def \W {\mathcal{W}}

\definecolor{mygray}{gray}{0.6}

\newcommand\graybullet{\textcolor{mygray}{\bullet}}

\long\def\symbolfootnote[#1]#2{\begingroup%
\def\thefootnote{\fnsymbol{footnote}}\footnote[#1]{#2}\endgroup}

\title{  \vspace {-2cm}\textbf{Simply generated non-crossing partitions}}
\date{}

\DeclareSymbolFont{extraup}{U}{zavm}{m}{n}
\DeclareMathSymbol{\varheart}{\mathalpha}{extraup}{86}
\DeclareMathSymbol{\vardiamond}{\mathalpha}{extraup}{87}

\makeatletter
\renewcommand*{\@fnsymbol}[1]{\ensuremath{\ifcase#1\or  \spadesuit \or \varheart\or \vardiamond \or \clubsuit \or
   \mathsection\or \mathparagraph\or \|\or **\or \dagger\dagger
   \or \ddagger\ddagger \else\@ctrerr\fi}}
\makeatother

\author{Igor Kortchemski\thanks{CNRS \& \'Ecole polytechnique. \hfill  \texttt{igor.kortchemski@normalesup.org}} 
\qquad \& \qquad Cyril Marzouk\thanks{Universit\"at Z\"urich.\hfill  \texttt{cyril.marzouk@math.uzh.ch}} 
}

\begin{document}

\maketitle

\let\thefootnote\relax\footnotetext{ \\ \emph{I.K.~acknowledges partial support from Agence Nationale de la Recherche, grant number ANR-14-CE25-0014 (ANR GRAAL), and from the City of Paris, grant “Emergences Paris 2013, Combinatoire à Paris” \\ C. M.~acknowledges support from the Swiss National Science Foundation 200021\_144325/1.}
\\ \\
\emph{MSC2010 subject classifications}. Primary 05C80, 60C05; secondary: 05C05, 60J80. }
 
\vspace {-0.5cm}

\begin{abstract} 
\medskip We introduce and study the model of simply generated non-crossing partitions, which are, roughly speaking, chosen at random according to a sequence of weights. This framework encompasses the particular case of uniform non-crossing partitions with constraints on their block sizes. Our main tool is a bijection between non-crossing partitions and plane trees, which maps such simply generated non-crossing partitions into simply generated trees so that blocks of size $k$ are in correspondence with vertices of outdegree $k$. This allows us to obtain limit theorems concerning the block structure of simply generated non-crossing partitions. We apply our results in free probability by giving a simple formula relating the maximum of the support of a compactly supported probability measure on the real line in term of its free cumulants.\end{abstract}

\section{Introduction}

We are interested in the structure of non-crossing partitions. The latter were introduced by Kreweras \cite {Kre72}, and quickly became a standard object in combinatorics. They have also appeared in many different other contexts, such as low-dimensional topology, geometric group theory and free probability (see e.g.~the survey \cite{McC06} and the references therein).  In this work, we study   combinatorial and geometric aspects of large \emph{random} non-crossing partitions.

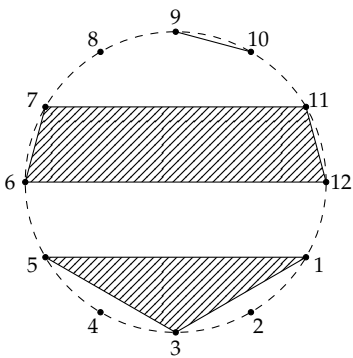
\begin{figure}[ht] \centering
\begin{scriptsize}
\begin{tikzpicture}
\draw[thin, dashed]	(0,0) circle (2);
\foreach \x in {1, 2, ..., 12}
	\coordinate (\x) at (-\x*360/12 : 2);
\foreach \x in {1, 2, ..., 12}
	\draw
	[fill=black]	(\x) circle (1pt)
	(-\x*360/12 : 2*1.1) node {\x}
;
\filldraw[pattern=north east lines]
	(1) -- (3) -- (5) -- cycle
	(6) -- (7) -- (11) -- (12) -- cycle
	(9) -- (10)
;
\end{tikzpicture}
\end{scriptsize}
\caption{The non-crossing partition $\{\{1, 3, 5\}, \{2\}, \{4\}, \{6, 7, 11, 12\}, \{8\}, \{9, 10\}\}$ of $[12]$.}
\label{fig:nc}
\end{figure}

Recall that a partition of $[n] \coloneqq \{1,2, \ldots,n\}$ is a collection of (pairwise) disjoint subsets, called blocks, whose
union is $[n]$. A non-crossing partition of $[n]$ is a partition of the vertices of a regular $n$-gon (labelled by the set $[n]$ in clockwise order) with the property that the convex hulls of its blocks are pairwise disjoint (see Fig.~\ref{fig:nc} for an example).

\paragraph{Large discrete combinatorial structures.} There are many ways to study discrete structures. Given a finite combinatorial class $ \mathcal{A}_{n}$ of objects of ``size'' $n$, a first step is often to calculate as explicitly as possible its cardinal $\# \mathcal{A}_{n}$, using for instance bijective arguments or generating functions. For non-crossing partitions, it is well-known that they are enumerated by Catalan numbers. It is also often of interest to enumerate elements of $ \mathcal{A}_{n}$ satisfying  constraints. For instance,   the number of non-crossing partitions of $[n]$ with given block sizes \cite{Kre72}, or the total number of blocks \cite{Edel80} have been studied. Edelman \cite{Edel80} also introduced and enumerated $k$-divisible non-crossing partitions (where all blocks must have size divisible $k$), which have also been studied by  Arizmendi \& Vargas \cite{AV12} in connection with free probability. Arizmendi \& Vargas also studied $k$-equal non-crossing partitions (where all blocks must have size exactly $k$).

{In probabilistic combinatorics}, one is interested in the properties of a \emph{typical} element of $ \mathcal{A}_{n}$. In other words, one studies statistics of a random element $\mathsf{a}_{n}$ of $ \mathcal{A}_{n}$ chosen uniformly at random. Graph theoretical properties of different uniform plane non-crossing structures obtained from a regular polygon have been considered in the past years. For example, \cite{DFHN99,GW00,DMN12,CKdissections} study the maximal degree in random triangulations, \cite{BPS10,CKdissections} obtain concentration bounds for the maximal degree in random dissections, and \cite{MP02,DN02,CKdissections} are interested in the structure of non-crossing trees. However, uniform non-crossing partitions have attracted less attention.  Arizmendi \cite{Ari12} finds the expected number of blocks of given size for non-crossing partitions of $[n]$ with certain constraints on the block sizes, Ortmann \cite{Ort12} shows that the distribution of a uniform random block in a uniform non-crossing partition $P_{n}$ of $[n]$ converges to a geometric random variable of parameter $1/2$ as $n \rightarrow \infty$ and limit theorems concerning the length of the longest chord of $P_{n}$ are obtained in \cite{CKdissections}.

It is also of interest to sample an element $\mathsf{a}_{n}$ of $\mathcal{A}_{n}$ according to a probability distribution different from the uniform law; one then studies the impact of this change on the asymptotic behavior of $ \mathsf{a}_{n}$ as $n \rightarrow \infty$. Certain families of probability distributions lead to the same asymptotic properties, and are said to belong the same universality class. However, the structure of  $ \mathsf{a}_{n}$ may drastically be impacted.To the best of our knowledge, only uniform non-crossing partitions have yet been studied in \cite{AV12,Ort12,CKdissections}.

Finally, another direction is to study distributional limits of $ \mathsf{a}_{n}$. Indeed, if it is possible to see the elements of the combinatorial class under consideration as elements of a same metric space, it makes sense to study the convergence in distribution of the sequence of random variables  $(\mathsf{a}_{n})_{n \geq 1}$ in this metric space. In the case of uniform non-crossing partitions, this approach has been followed in \cite{CKdissections} by seeing them as compact subsets of the unit disk; we extend the result obtained there to simply generated non-crossing partitions.

\paragraph{Simply generated non-crossing partitions.} In this work, we propose to sample non-crossing partitions at random according to a Boltzmann-type distribution, which depends on a sequence of weights. For every integer $n \geq 1$, denote by $ \NC_{n}$ the set of all non-crossing partitions of $[n]$; given a sequence of non-negative real numbers $w = (w(i); i \ge 1)$, with every partition $P \in \NC_{n}$, we associate a weight $\Omega^w(P)$:
$$\Omega^w(P)= \prod_{B \textrm{ block of } P} w( \textrm{size of }B).$$
Then, for every $ P \in \NC_{n}$, set
$$ \P^w_{n}(P)= \frac{ \Omega^w(P)}{ \sum_{Q \in \NC_{n}} \Omega^w(Q)}.$$
Implicitly, we shall always restrict our attention to those values of $n$ for which $ \sum_{P \in \NC_{n}} \Omega^w(P)>0$.
A random non-crossing partition of $[n]$ sampled according to $ \P^w_{n}$ is called a \emph{simply generated non-crossing partition}. We chose this terminology because of the similarity with the model of simply generated trees, introduced by Meir \& Moon \cite{MM78} and whose definition we recall in Sec.~\ref{sec:sim} below. We were also inspired by recent work on scaling limits of Boltzmann-type random graphs \cite{LGM11,Kor11}.

We point out that, taking $w(i)=1$ for every $i \ge 1$, $\P^w_n$ is the uniform distribution on $\NC_n$; more generally, if $ \mathcal{A}$ is a non-empty subset of $\N=\{ 1,2,3, \ldots\}$, and $ w_{\mathcal{A} }(i)=1$ if $i \in \mathcal{A}$ and $ w_{ \mathcal{A} }(i)=0$ if $i \not \in \mathcal{A}$, then $ \P^ {w_{ \mathcal{A} }}_{n}$ is the uniform distribution on the subset of $\NC_n$ formed by  partitions with all block sizes belonging to $ \mathcal{A}$ (provided that they exist), and which we call $ \mathcal{A}$-constrained non-crossing partitions. In particular, by taking $ \mathcal{A}= \{k\}$ one gets uniform $k$-equal non-crossing partitions, and by taking $ \mathcal{A}= k \mathbb{N}$ one gets  uniform $k$-divisible non-crossing partitions.

\paragraph{Bijections between non-crossing partitions and plane trees.} Our main tools to study simply generated non-crossing partitions are bijections with plane trees. We explain here the main ideas, and refer to Sec.~\ref{sec:pp} for details. With a non-crossing partition, we start by associating a (two-type) dual tree, as depicted in Fig.~\ref{fig:ex_part}.

\begin{figure}[ht] \centering
\begin{scriptsize}
\begin{tikzpicture}
\draw[thin, dashed]	(0,0) circle (2);
\foreach \x in {1, 2, ..., 12}
	\coordinate (\x) at (-\x*360/12 : 2);
\foreach \x in {1, 2, ..., 12}
	\draw
	[fill=black]	(\x) circle (1pt)
	(-\x*360/12 : 2*1.1) node {\x}
;
\filldraw[pattern=north east lines]
	(1) -- (3) -- (5) -- cycle
	(6) -- (7) -- (11) -- (12) -- cycle
	(9) -- (10)
;
\end{tikzpicture}
\qquad\qquad
%
\begin{tikzpicture}
\draw[thin, dashed]	(0,0) circle (2);
\foreach \x in {1, 2, ..., 12}
	\coordinate (\x) at (-\x*360/12 : 2);
\foreach \x in {1, 2, ..., 12}
	\draw	(-\x*360/12 : 2*1.1) node {\x}
;
\draw[dotted]
	(1) -- (3) -- (5) -- cycle
	(6) -- (7) -- (11) -- (12) -- cycle
	(9) -- (10)
;
%
\coordinate (A) at (0, .4);
\coordinate (B) at (0, -1.3);
\coordinate (C) at (2);
\coordinate (D) at (4);
\coordinate (E) at (8);
\coordinate (F) at ($(9)!0.5!(10)$);
%
\coordinate (A') at (0, -.5);
\coordinate (B') at (360/48-2*360/12 : 2*.92);
\coordinate (C') at (3*360/48-5*360/12 : 2*.92);
\coordinate (D') at (360/24-7*360/12 : 2*1.1);
\coordinate (E') at (0, 1.25);
\coordinate (F') at (360/24-10*360/12 : 2*1.1);
\coordinate (G') at (360/24-12*360/12 : 2*1.1);
\draw
	(A) -- (A') -- (B)
	(B) -- (B') -- (C)
	(B) -- (C') -- (D)
	(A) -- (D')
	(A) -- (E')
	(E') -- (E)
	(E') -- (F) -- (F')
	(A) -- (G')
;
\foreach \x in {A, B, ..., F}
	\draw[fill=black]	(\x) circle (1.5pt)
;
\foreach \x in {A, B, ..., G}
	\draw[fill=white]	(\x') circle (1.5pt)
;
\end{tikzpicture}
\end{scriptsize}
\caption{The (non-crossing) partition $\{\{1, 3, 5\}, \{2\}, \{4\}, \{6, 7, 11, 12\}, \{8\}, \{9, 10\}\}$ and its dual tree.}
\label{fig:ex_part}
\end{figure}
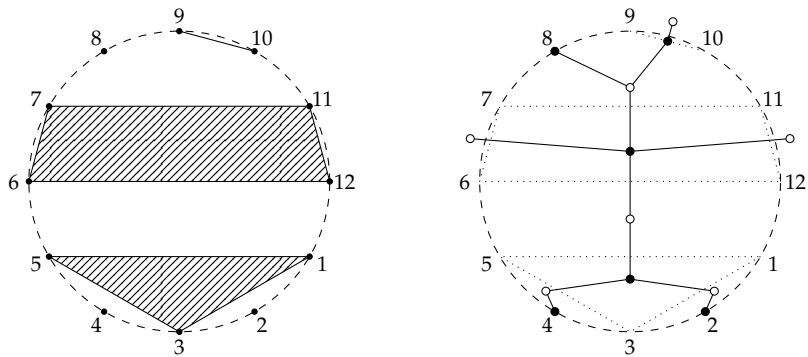
We choose an appropriate root for this two-type tree, and then apply a recent bijection due to Janson \& Stef{\'a}nsson \cite{JS12}; this yields a bijection $\mathcal{B}^{\circ}$ between $\NC_{n}$ and plane trees with $n+1$ vertices. We mention here that this bijection was directly defined  by  Dershowitz \& Zaks \cite{DZ86} without using the dual two-type tree. It turns out that other known bijections between non-crossing partitions and plane trees, such as  Prodinger's bijection \cite{Pro83} and  the Kreweras complement \cite{Kre72}, can be obtained by choosing to distinguishing another root in the dual two-type tree (again see Sec.~\ref{sec:pp} below for details). Our contribution is therefore to unify previously known bijections between non-crossing partitions and plane trees by showing that they all amount to doing certain operations on the dual tree of a non-crossing partition, and to use them to study \emph{random} non-crossing partitions.

It turns out that the dual tree of a simply generated non-crossing partition is a two-type simply generated tree (Prop.~\ref{prop:PNC_et_arbre_deux_types}). A crucial feature of the bijection  $\mathcal{B}^{\circ}$ it that it maps simply generated non-crossing partitions into simply generated trees in such a way that blocks of size $k$ are in correspondence with vertices with outdegree $k$ (Prop.~\ref{prop:bij}). This allows to reformulate questions on simply generated non-crossing partitions involving block sizes in terms of simply generated trees involving outdegrees. The point is that the study of simply generated trees is a well-paved road. In particular, this allows us to show that if $P_{n}$ is a simply generated non-crossing plane partition of $[n]$, then, under certain conditions, the size of a block chosen uniformly at random in $P_{n}$ converges in distribution as $n \rightarrow \infty$ to an explicit probability distribution depending on the weights. We also obtain, for a certain family of weights, asymptotic normality of the block sizes and limit theorems for the sizes of the largest blocks.  We specify here some of these results for $ \mathcal{A}$-constrained non-crossing partitions, and refer to Section \ref{sec:appli} for more general statements and further applications.

\begin{thm}\label{thm:blocks}Let $ \mathcal{A}$ be a non-empty subset of $\N$ with $ \mathcal{A} \neq \{1\}$, and let  $P^{\mathcal{A} }_{n}$ be a random non-crossing partition chosen uniformly at random among all those with block sizes belonging to $ \mathcal{A}$ (provided that they exist).  Let  $\pi_{ \mathcal{A} }$ be the probability measure on $\Z_{+}= \{0,1,2,\ldots\} $ defined by $$ \pi_{\mathcal{A}}(k)= \frac{ \xi_{ \mathcal{A} }^{k}}{ {1+}\sum_{i \in \mathcal{A}}  \xi_{ \mathcal{A} }^{i} } \mathbbm {1}_{k \in \{0\} \cup \mathcal{A}}, \qquad \textrm {where }  \xi_{ \mathcal{A} }>0 \textrm { is such that } \qquad  1+\sum_{i \in \mathcal{A} }  \xi_{ \mathcal{A} }^{i}= \sum_{i \in \mathcal{A} } i  \cdot  \xi_{ \mathcal{A} }^{i}.$$
\begin{enumerate}
\item[(i)] Let $S_{1}(P^{\mathcal{A} }_{n})$ be the size of the block containing $1$ in $P^{\mathcal{A} }_{n}$. Then, for every $k \geq 1$, $ \Pr{S_{1}(P^{\mathcal{A} }_{n})=k} \rightarrow k \pi_{ \mathcal{A} }(k)$ as $n \rightarrow \infty$.
\item[(ii)] Let $B_{n}$ be a block chosen uniformly at random in $P^{\mathcal{A} }_{n}$. Then, for every $k \geq 1$, $\Pr{|B_{n}|=k} \rightarrow \pi_{ \mathcal{A} }(k)/(1-\pi_{ \mathcal{A} }(0))$ as $n \rightarrow  \infty$.
\item[(iii)] Let $C$ be a non-empty subset of $\N$ and denote by $ \zeta_{C}(P^{\mathcal{A} }_{n})$  the number of blocks of $P^{\mathcal{A} }_{n}$ whose size belongs to $C$.
As $n \rightarrow  \infty$, the convergence $ \zeta_{C}(P^{\mathcal{A} }_{n})/n \rightarrow  \pi_{ \mathcal{A} }(C)$ holds in probability and, in addition, $\Es{\zeta_{C}(P^{\mathcal{A} }_{n})}/n \rightarrow  \pi_{ \mathcal{A} }(C)$.\end{enumerate}
\end{thm}
In the particular case of uniform $k$-divisible non-crossing partitions, Theorem \ref{thm:blocks} (ii,iii) has been obtained by Ortmann \cite[Sec.~2.3]{Ort12}. Also, Arizmendi \cite{Ari12} obtained  by combinatorial means closed formulas for the expected number of blocks of given size in $k$-divisible non-crossing partitions.

\paragraph{Applications in free probability.} An additional motivation for introducing simply generated non-crossing partitions comes from free probability. Indeed, the partition function$$ Z^w_{n} \coloneqq \sum_{P \in \NC_{n}} \prod_{B \textrm{ block of } P} w( \textrm{size of }B)$$
expresses the moments of a measure  in terms of its free cumulants. More precisely, if $  \mu$ is a probability measure on $ \mathbb{R}$ with compact support, its Cauchy transform
$$  G_{\mu}(z)=  \int_{ \mathbb {R}} \frac{ \mu(dt)}{z-t}, \qquad z \in \mathbb {C} \setminus \mathrm{supp}\, \mu$$
is analytic and locally invertible on a neighbourhood of $ \infty$; its inverse $K_{\mu}$ is meromorphic around zero, with a simple pole of residue $1$ (see e.g.~\cite[Sec.~5]{BV93}). One can then write
$$ R_{\mu}(z)=K_{\mu}(z)- \frac{1}{z}= \sum_{n=0}^{\infty}\kappa_{n+1}(\mu) z^{n}.$$
The analytic function $R_{ \mu}$ is called the $R$-transform of $ \mu$, and uniquely defines $ \mu$. In addition, the coefficients $(\kappa_{n}(\mu); n \geq 1)$ are called the free cumulants of $ \mu$. The importance of $R$-transforms stems in the fact that they linearize free additive convolution and characterize weak convergence of probability measures, see \cite {BV93}. The following relation between the moments of $ \mu$ and its free cumulants is a well-known fact, that goes up to \cite {Spe94}. Let $  \mu$ be a compactly supported probability measure on $ \mathbb{R}$. Then, for every $n \geq 1$,
\begin{equation}
\label{eq:part}\int_{ \mathbb{R}} t^{n} \mu(dt) =  \sum_{P \in \NC_{n}} \prod_{B \textrm{ block of } P} \kappa_{ \textrm {size}(B)}(\mu).
\end{equation}
In other words, the $n$-th moment of $ \mu$ is the partition function of simply generated non-crossing partitions on $[n]$ with weights $ w(i)=\kappa_{i}(\mu)$ given by the free cumulants of $ \mu$. Using the bijection $\mathcal{B}^{\circ} $, we establish the following result.

\begin {thm}\label {thm:support}Let $  \mu$ be a compactly supported probability measure on $ \mathbb{R}$, different from a Dirac mass, and such that  all its free cumulants $(\kappa_{i}(\mu); i \geq 1)$ are nonnegative. Let $s_{\mu}$ be the maximum of its support. Set $$\rho = \left(  \limsup_{n \rightarrow \infty} \kappa_{n}(\mu)^{1/n} \right)^{-1} \qquad \textrm {and}\qquad  \nu =1+ \lim_{t \uparrow \rho}\frac{t^{2} R_{\mu}'(t)-1 }{t R_{\mu}(t)+1}.$$
If $ \nu \geq 1$, there exists a unique number $ \xi$ in $(0,\rho]$ such that $R'_{\mu}(\xi)=1/\xi^{2}$, and, in addition, 
$$ s_{\mu} \quad =  \quad \begin {cases}   
  \frac{1}{\xi}+  R_{\mu}(\xi)   & \textrm { if } \nu \geq 1, \\
   \frac{1}{\rho}+  R_{\mu}(\rho)   & \textrm { if } \nu<1. \end {cases}$$
\end {thm}

See Sec.~\ref{sec:free} for examples.
This gives a more explicit formula that the one obtained by Ortmann \cite[Thm.~5.4]{Ort12}, which reads
$$  \log(s_{\mu})= \sup \left\{ \frac{1}{m_{1}(p)}  \sum_{n \in L} p_{n}  \log \left(  \frac{\kappa_{n}(\mu)}{p_{n}} \right)   - \frac{ \theta(m_{1}(p))}{m_{1}(p)} ; \ p \in \mathfrak {M}_{1}^{1}(L)\right\},$$
where $L= \{ n \geq 1;  \kappa_{n}(\mu) \neq 0\}$, $ \theta(x)= \log(x-1)- x \log(x-1/x)$, $\mathfrak {M}_{1}^{1}(L)$ is the set of probability measures $p = (p_n ; n \in \N)$ on $ \N$ with $p(L^{c})=0$ and $m_{1}(p)$ is the mean of $p$.

\paragraph{Non-crossing partitions seen as compact subsets of the unit disk.} Finally, if $P_{n}$ is a simply generated non-crossing partition of $[n]$, we study the distributional limits of $P_{n}$, seen as compact subset of the unit disk by identifying each integer $l \in [n]$ with the complex number ${\rm e}^{-2{\rm i}\pi l/n}$. This route was followed in \cite{CKdissections}, where it was shown that as $n \rightarrow  \infty$, a uniform non-crossing partition of $[n]$ converges in distribution to Aldous' Brownian triangulation of the disk \cite{Ald94b}, in the space of all compact subsets of the unit disk equipped with the Hausdorff metric, and where the Brownian triangulation is a random compact subset of the unit disk constructed from the Brownian excursion. We show more generally that a whole family of simply generated non-crossing partitions of $[n]$ (including uniform $ \mathcal{A}$-constrained non-crossing partitions) converge in distribution to the Brownian triangulation, and show that other families converge in distribution to the stable lamination, which is another random compact subset of the unit disk introduced in \cite{Kor11}.  We refer to Sec.~\ref{sec:Hausdorff} for details and precise statements. 

 \begin{figure}[!h]
 \begin{center}
    \includegraphics[width=0.3 \linewidth]{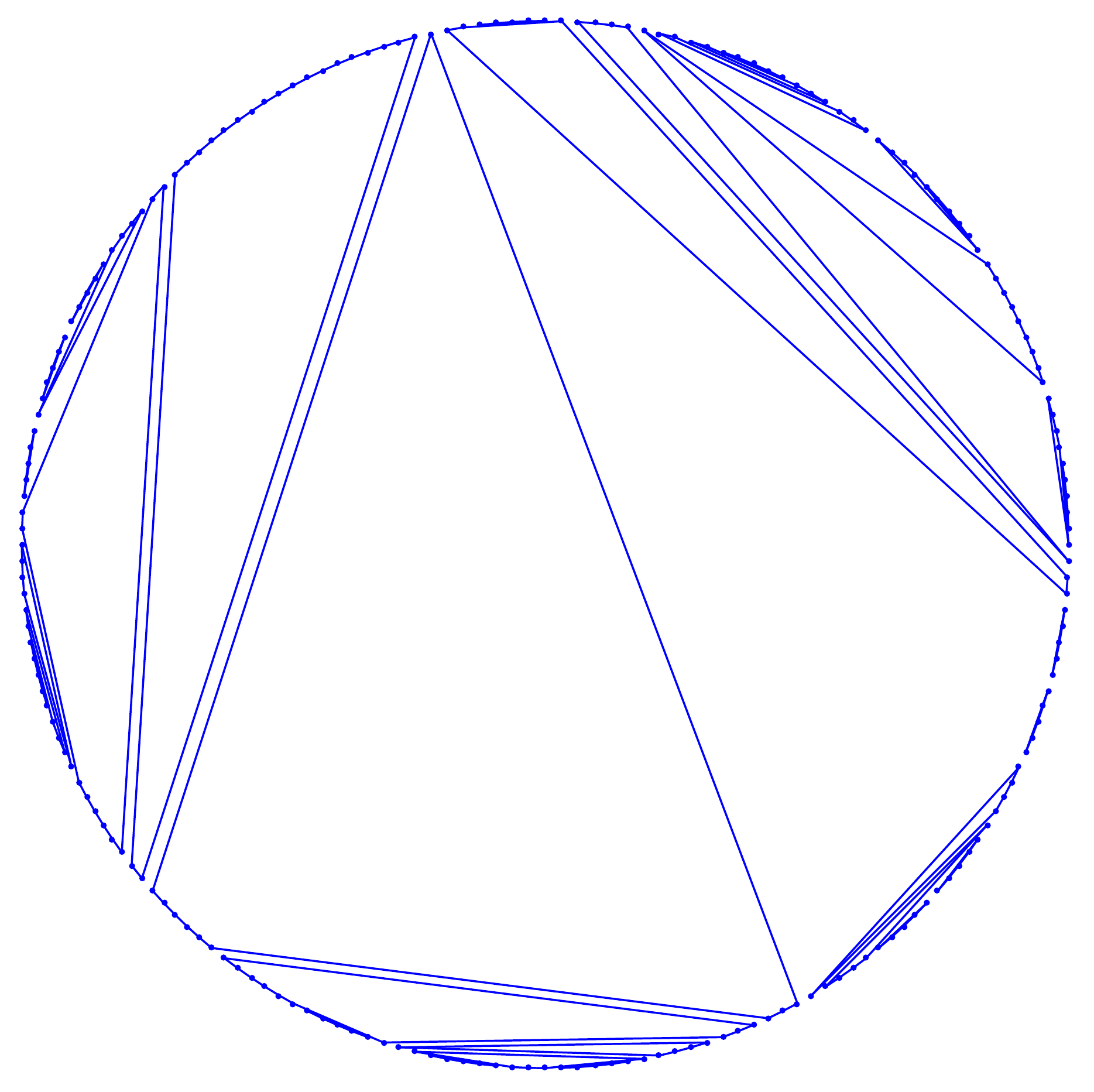}\hfill
    \includegraphics[width=0.3 \linewidth]{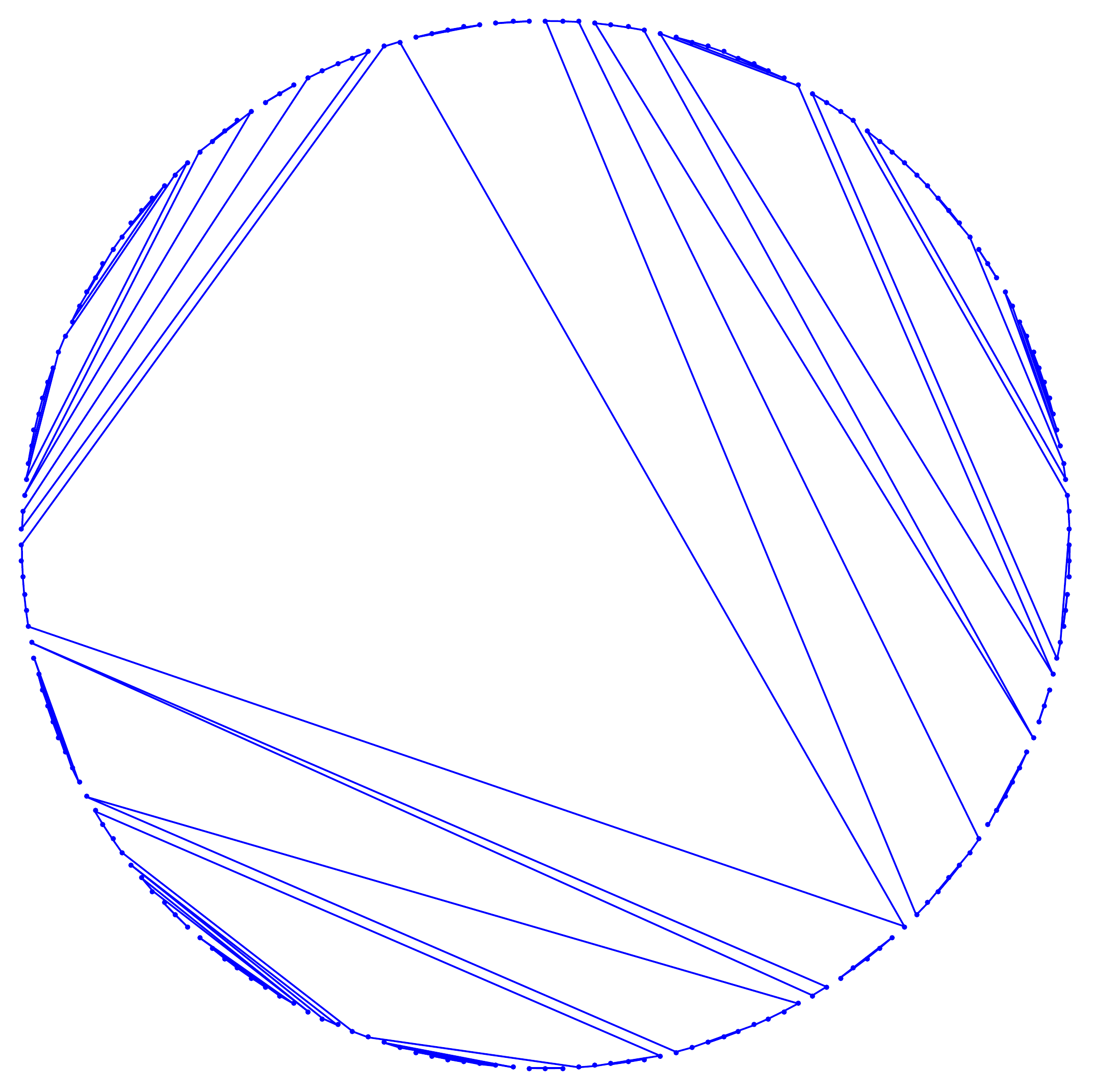}\hfill
  \includegraphics[width=0.3 \linewidth]{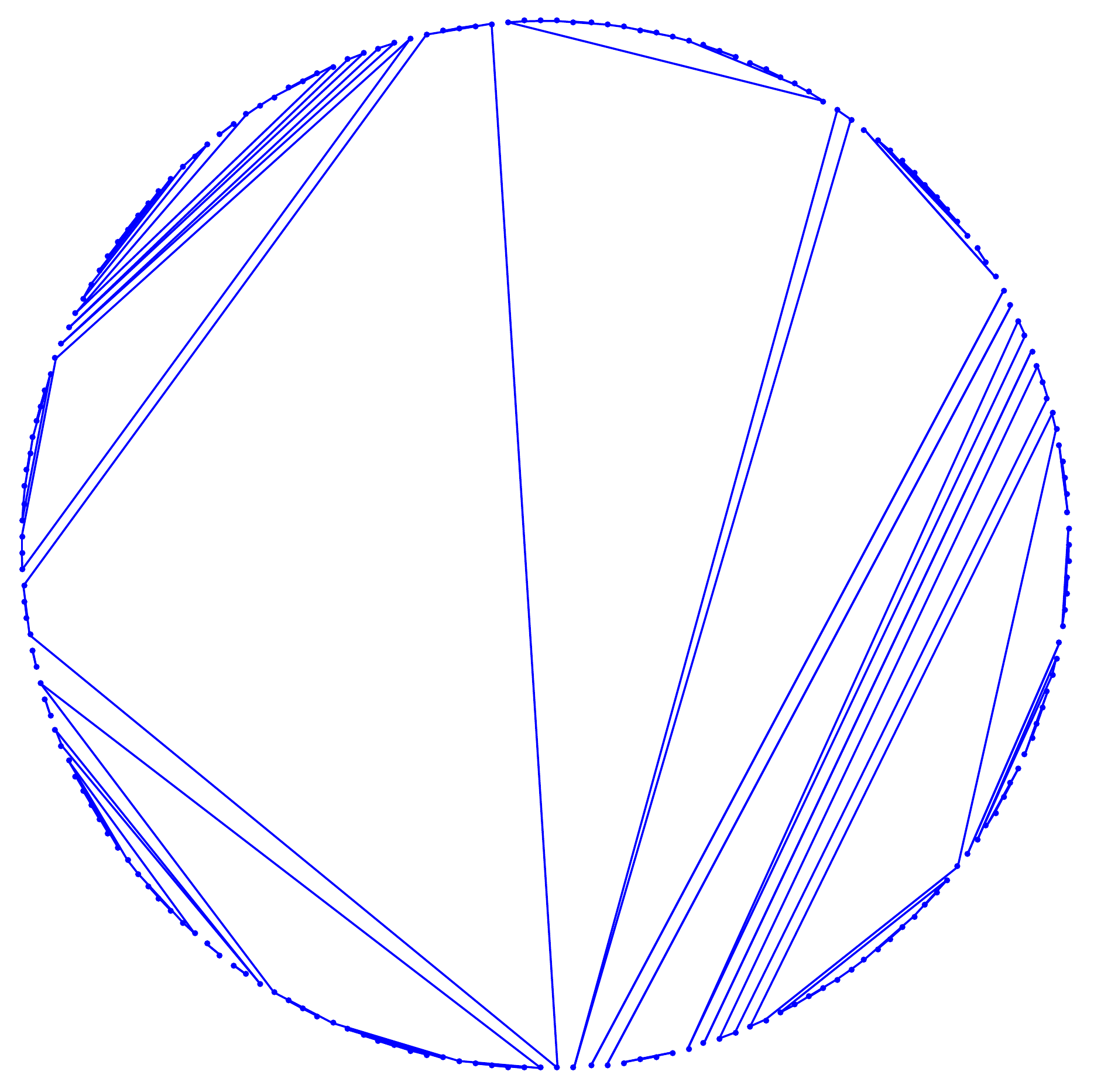}
   \caption{ \label{fig:largepnc} Simulations of random non-crossing partitions of $[200]$ chosen uniformly at random among all those having respectively only block sizes that are multiples of $5$, block sizes that are odd and block sizes that are prime numbers.}
 \end{center}
 \end{figure}
 This has in particular applications concerning the length of the longest chord of $P_{n}$. By definition, the (angular) length of a chord  $[e^{-2i\pi s},e^{-2i\pi t}]$ with $0 \leq s \leq t \leq 1$ is $\min(t-s,1-t+s)$. Denote by $\mathsf{C}(P_{n})$ the length of the longest chord of $P_{n}$.  In the case of $ \mathcal{A}$-constrained non-crossing partitions, we prove in particular the following result.

\begin{thm}\label{thm:A}Let $ \mathcal{A}$ is a non-empty subset of $\N$ with $ \mathcal{A} \neq \{1\}$, and let  $P^{\mathcal{A} }_{n}$ be a random non-crossing partition chosen uniformly at random among all those with block sizes belonging to $ \mathcal{A}$ (provided that they exist).  Then, as $n \rightarrow \infty$, $\mathsf{C}(P^{\mathcal{A}}_{n})$ converges in distribution to a random variable with distribution
$$ \frac{1}{\pi} \frac{3x-1}{ x^{2}(1-x)^{2} \sqrt{1-2x}} \mathbbm{1}_{ \frac{1}{3} \leq x \leq \frac{1}{2}} dx.$$
\end{thm}
It is remarkable that the limiting distribution in Theorem \ref{thm:A} does not depend on $ \mathcal{A}$ (it seems that this is \emph{not} the case for the largest block area, see Section \ref{sec:ext}).  

This bears some similarity with \cite{CKdissections}, but we emphasize that this is not a simple adaptation of the arguments of \cite{CKdissections}. Indeed, roughly speaking, \cite{CKdissections} manages to code uniform non-crossing partitions of $[n]$ by a dual-type uniform plane tree. In the more general case of simply generated non-crossing partitions, the dual tree is a more complicated two-type tree and the  Janson--Stef{\'a}nsson bijection is needed.

\paragraph{Acknowledgments.} I.K.~is grateful to Octavio Arizmendi for introducing him to $k$-divisible non-crossing partitions during a stay at CIMAT, and to the University of Z\"urich, where this work began, for its hospitality. Both authors thank Valentin Féray for a stimulating discussion, and the Newton Institute for its hospitality, where this work was finished.

\section{Bijections between non-crossing partitions and plane trees}
\label{section:bijection_NC_arbres_plans}

We denote by $\D = \{z \in \C : |z| < 1\}$ the open unit disk of the complex plane, by $\S = \{z \in \C : |z| = 1\}$ the unit circle and by $\overline{\D} = \D \cup \S$ the closed unit disk. For every $x, y \in \S$, we write $[x,y]$ for the line segment between $x$ and $y$ in $\overline{\D}$, with the convention $[x,x] = \{x\}$. A geodesic lamination $L$ of $\overline{\D}$ is a closed subset of $\overline{\D}$ which can be written as the union of a collection of non-crossing such chords, i.e. which do not intersect in $\D$. In this paper, by lamination we will always mean geodesic lamination of $\overline{\D}$.

We view a partition of $[n]$ as a closed subset of $\overline{\D}$ by identifying each integer $l \in [n]$ with the complex number ${\rm e}^{-2{\rm i}\pi l/n}$ and by drawing a chord $[{\rm e}^{-2{\rm i}\pi l/n}, {\rm e}^{-2{\rm i}\pi l'/n}]$ whenever $l, l' \in [n]$ are two consecutive elements of the same block of the partition, where the smallest and the largest element of a block are consecutive by convention. The partition is non-crossing if and only if these chords do not cross; we implicitly identify a non-crossing partition with the associated lamination throughout this paper.

Let $\T$ be the set of all finite plane trees (see the definition below), and $\T_{n}$ be the set of all plane trees with $n$ vertices. We construct two bijections between $\NC_n$ and $\T_{n+1}$. The study of a (random) non-crossing partition then reduces to that of the associated (random) plane tree.

\subsection{Non-crossing partitions and plane trees}
\label{sec:pp}

Recall that $\N = \{1, 2, \dots\}$ is the set of all positive integers, set $\N^0 = \{\varnothing\}$ and let
\begin{equation*}
\mathcal{U} = \bigcup_{n \ge 0} \N^n.
\end{equation*}
For $u = (u_1, \dots, u_n) \in \mathcal{U}$, we denote by $|u| = n$ the length of $u$; if $n \ge 1$, we define $pr(u) = (u_1, \dots, u_{n-1})$ and for $i \ge 1$, we let $ui = (u_1, \dots, u_n, i)$; more generally, for $v = (v_1, \dots, v_m) \in \mathcal{U}$, we let $uv = (u_1, \dots, u_n, v_1, \dots, v_m) \in \mathcal{U}$ be the concatenation of $u$ and $v$. We endow $\mathcal{U}$ with the lexicographical order: $v \prec w$ if there exists $z \in \mathcal{U}$ such that $v = z(v_1, \dots, v_n)$, $w = z(w_1, \dots, w_m)$ and $v_1 < w_1$.

A plane tree is a nonempty, finite subset $\tau \subset \mathcal{U}$ such that:
\begin{enumerate}
\item $\varnothing \in \tau$;
\item if $u \in \tau$ with $|u| \ge 1$, then $pr(u) \in \tau$;
\item if $u \in \tau$, then there exists an integer $k_u \ge 0$ such that $ui \in \tau$ if and only if $1 \le i \le k_u$.
\end{enumerate}
We will view each vertex $u$ of a tree $\tau$ as an individual of a population whose $\tau$ is the genealogical tree. The vertex $\varnothing$ is called the root of the tree and for every $u \in \tau$, $k_u$ is the number of children (or outdegree) of $u$ (if $k_u = 0$, then $u$ is called a leaf), $|u|$ is its generation, $pr(u)$ is its parent and more generally, the vertices $u, pr(u), pr \circ pr (u), \dots, pr^{|u|}(u) = \varnothing$ are its ancestors. In the sequel, if non specified otherwise, by tree we will always mean plane tree.

We define the (planar, but non-rooted) dual tree $T(P)$ of a non-crossing partition $P$ of $[n]$ as follows: we place a black vertex inside each block of the partition and a white vertex inside each other face, then we join two vertices if the corresponding faces share a common edge; here we shall view the singletons as self-loops and the blocks of size two with one double edge. See Fig. \ref{fig:ex_part} for an illustration. Observe that the graph thus obtained is a indeed a planar tree (meaning that
there is an order among all edges adjacent to a same vertex, up to cyclic permutations), with $n+1$ vertices, and that the latter is bipartite: each edge connects two vertices of different colours.

In order to fully recover the partition from the tree (and therefore obtain a bijection), we need to assign a root by distinguishing a corner of $T(P)$ (a corner of a vertex in a planar tree is a sector around this vertex delimited by two consecutive edges), thus making it a plane tree. We will do so in two different ways, which will give rise to two different bijections. First, $T^\circ(P)$ is the tree $T(P)$ rooted at the corner of the white vertex that lies in the face containing the vertices $1$ and $n$, and that has the black vertex in the block containing $1$ as its first child; $T^\bullet(P)$ is the tree $T(P)$ rooted at the corner of the black vertex in the block containing $n$ and that has the white vertex that lies in the face containing the vertices $1$ and $n$ as its first child; see Fig.~\ref{fig:arbre_GW} and \ref{fig:arbre_non_GW} for an example.

The trees $T^\circ(P)$ and $T^\bullet(P)$ are two-type plane trees: vertices at even generation are coloured in one colour and vertices at odd generation are coloured in another colour. We apply to each a bijection due to Janson \& Stef{\'a}nsson \cite[Sec.~3]{JS12} which maps such a tree into a one-type tree, that we now describe. This bijection enjoys useful probabilistic features, see Corollary \ref{cor:bijection_JS_arbres_simplement_generes} below.

We denote by $T$ a plane tree and by $\mathcal{G}(T)$ its image by this bijection; $T$ and $\mathcal{G}(T)$ have the same vertices but the edges are different. If $T = \{\varnothing\}$ is a singleton, then set $\mathcal{G}(T) = \{\varnothing\}$; otherwise, for every vertex $u \in T$ at even generation with $k_u \ge 1$ children, do the following: first, if $u \ne \varnothing$, draw an edge between its parent $pr(u)$ and its first child $u1$, then draw edges between its consecutive children $u1$ and $u2$, $u2$ and $u3$, ..., $u(k_u-1)$ and $uk_u$, and finally draw an edge between $uk_u$ and $u$; if $u$ is a leaf of $T$, then this procedure reduces to drawing an edge between $u$ and $pr(u)$. We root $\mathcal{G}(T)$ at the first child of the root of $T$. One can check that $\mathcal{G}(T)$ thus defined is indeed a plane tree, and that the mapping is invertible. Also observe that every vertex at even generation in $T$ is mapped to a leaf of $\mathcal{G}(T)$, and every vertex at odd generation with $k \ge 0$ children in $T$ is mapped to a vertex with $k + 1$ children in $\mathcal{G}(T)$.

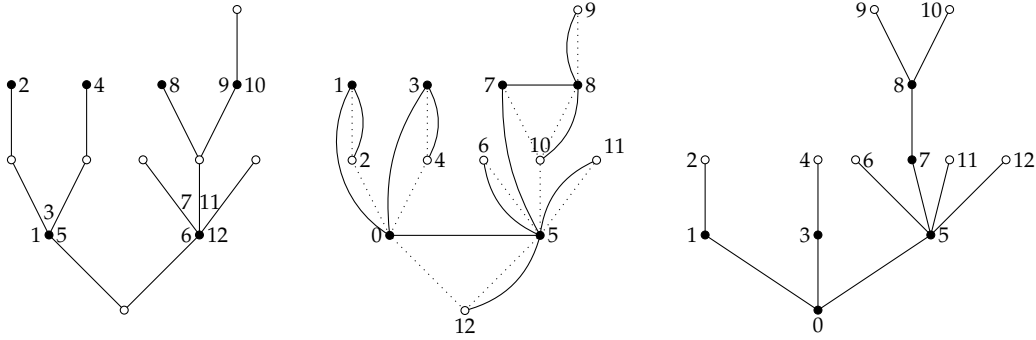
\begin{figure}[h!] \centering
\begin{scriptsize}
\begin{tikzpicture}
\coordinate (0) at (0,0);
	\coordinate (1) at (-1,1);
		\coordinate (11) at (-1.5,2);
			\coordinate (111) at (-1.5,3);
		\coordinate (12) at (-.5,2);
			\coordinate (121) at (-.5,3);
	\coordinate (2) at (1,1);
		\coordinate (21) at (.25,2);
		\coordinate (22) at (1,2);
			\coordinate (221) at (.5,3);
			\coordinate (222) at (1.5,3);
				\coordinate (2221) at (1.5,4);
		\coordinate (23) at (1.75,2);
\draw
	(0) -- (1)	(0) -- (2)
	(1) -- (11) -- (111)
	(1) -- (12) -- (121)
	(2) -- (21)	(2) -- (22)	(2) -- (23)
	(22) -- (221)	(22) -- (222) -- (2221)
;
\draw[fill=white]
	(0) circle (1.5pt)
	(11) circle (1.5pt)
	(12) circle (1.5pt)
	(21) circle (1.5pt)
	(22) circle (1.5pt)
	(23) circle (1.5pt)
	(2221) circle (1.5pt)
;
\draw[fill=black]
	(1) circle (1.5pt)
	(111) circle (1.5pt)
	(121) circle (1.5pt)
	(2) circle (1.5pt)
	(221) circle (1.5pt)
	(222) circle (1.5pt)
;
%
\draw
	(1) node[left] {$1$}
	(111) node[right] {$2$}
	(1)++(0,.1) node[above] {$3$}
	(121) node[right] {$4$}
	(1) node[right] {$5$}
	(2) node[left] {$6$}
	(2)++(0,.25) node[above left] {$7$}
	(221) node[right] {$8$}
	(222) node[left] {$9$}
	(222) node[right] {$10$}
	(2)++(-0.1,.25) node[above right] {$11$}
	(2) node[right] {$12$}
	(0) node[below, color=white] {$0$}
;
\end{tikzpicture}
\qquad
%
\begin{tikzpicture}
\coordinate (0) at (0,0);
	\coordinate (1) at (-1,1);
		\coordinate (11) at (-1.5,2);
			\coordinate (111) at (-1.5,3);
		\coordinate (12) at (-.5,2);
			\coordinate (121) at (-.5,3);
	\coordinate (2) at (1,1);
		\coordinate (21) at (.25,2);
		\coordinate (22) at (1,2);
			\coordinate (221) at (.5,3);
			\coordinate (222) at (1.5,3);
				\coordinate (2221) at (1.5,4);
		\coordinate (23) at (1.75,2);
\draw[dotted]
	(0) -- (1)	(0) -- (2)
	(1) -- (11) -- (111)
	(1) -- (12) -- (121)
	(2) -- (21)	(2) -- (22)	(2) -- (23)
	(22) -- (221)	(22) -- (222) -- (2221)
;
\draw
	(1) -- (2) to[bend left] (0)
	(1) to[bend left=45] (111) to[bend left] (11)
	(1) to[bend left=20] (121) to[bend left] (12)
	(2) to[bend left] (21)
	(2) to[bend left=15] (221) -- (222) to[bend left] (22)
	(2) to[bend left] (23)
	(222) to[bend left] (2221)
;
\draw[fill=white]
	(0) circle (1.5pt)
	(11) circle (1.5pt)
	(12) circle (1.5pt)
	(21) circle (1.5pt)
	(22) circle (1.5pt)
	(23) circle (1.5pt)
	(2221) circle (1.5pt)
;
\draw[fill=black]
	(1) circle (1.5pt)
	(111) circle (1.5pt)
	(121) circle (1.5pt)
	(2) circle (1.5pt)
	(221) circle (1.5pt)
	(222) circle (1.5pt)
;
%
\draw
	(1) node[left] {$0$}
	(111) node[left] {$1$}
	(11) node[right] {$2$}
	(121) node[left] {$3$}
	(12) node[right] {$4$}
	(2) node[right] {$5$}
	(21) node[above] {$6$}
	(221) node[left] {$7$}
	(222) node[right] {$8$}
	(2221) node[right] {$9$}
	(22) node[above] {$10$}
	(23) node[above right] {$11$}
	(0) node[below] {$12$}
;
\end{tikzpicture}
\qquad
%
\begin{tikzpicture}
\coordinate (0) at (0,0);
	\coordinate (1) at (-1.5,1);
		\coordinate (11) at (-1.5,2);
	\coordinate (2) at (0,1);
		\coordinate (21) at (0,2);
	\coordinate (3) at (1.5,1);
		\coordinate (31) at (.5,2);
		\coordinate (32) at (1.25,2);
			\coordinate (321) at (1.25,3);
				\coordinate (3211) at (.75,4);
				\coordinate (3212) at (1.75,4);
		\coordinate (33) at (1.75,2);
		\coordinate (34) at (2.5,2);
\draw
	(0) -- (1) -- (11)
	(0) -- (2) -- (21)
	(0) -- (3)
	(3) -- (31)	(3) -- (32)	(3) -- (33)	(3) -- (34)
	(32) -- (321) -- (3211)	(321) -- (3212)
;
\draw[fill=black]
	(0) circle (1.5pt)
	(1) circle (1.5pt)
	(2) circle (1.5pt)
	(3) circle (1.5pt)
	(32) circle (1.5pt)
	(321) circle (1.5pt)
;
\draw[fill=white]
	(11) circle (1.5pt)
	(21) circle (1.5pt)
	(31) circle (1.5pt)
	(33) circle (1.5pt)
	(34) circle (1.5pt)
	(3211) circle (1.5pt)
	(3212) circle (1.5pt)
;
%
\draw
	(0) node[below] {$0$}
	(1) node[left] {$1$}
	(11) node[left] {$2$}
	(2) node[left] {$3$}
	(21) node[left] {$4$}
	(3) node[right] {$5$}
	(31) node[right] {$6$}
	(32) node[right] {$7$}
	(321) node[left] {$8$}
	(3211) node[left] {$9$}
	(3212) node[left] {$10$}
	(33) node[right] {$11$}
	(34) node[right] {$12$}
;
\end{tikzpicture}
\end{scriptsize}
\caption{The tree $T^\circ$ associated with the partition from Fig. \ref{fig:ex_part}, with its black corners indexed according to the contour sequence, and its image $\mathcal{T}^\circ$ by the Janson--Stef{\'a}nsson bijection, with its vertices indexed in lexicographical order.}
\label{fig:arbre_GW}
\end{figure}
%
%
\begin{figure}[h!] \centering
\begin{scriptsize}
\begin{tikzpicture}
\coordinate (0) at (0,0);
	\coordinate (1) at (-1.5,1);
		\coordinate (11) at (-1.5,2);
			\coordinate (111) at (-2,3);
				\coordinate (1111) at (-2,4);
			\coordinate (112) at (-1,3);
				\coordinate (1121) at (-1,4);
	\coordinate (2) at (-.5,1);
	\coordinate (3) at (.5,1);
		\coordinate (31) at (0,2);
		\coordinate (32) at (1,2);
			\coordinate (321) at (1,3);
	\coordinate (4) at (1.5,1);
\draw
	(0) -- (1)	(0) -- (2)	(0) -- (3)	(0) -- (4)
	(1) -- (11)
	(11) -- (111) -- (1111)	(11) -- (112) -- (1121)
	(3) -- (31)	(3) -- (32) -- (321)
;
\draw[fill=black]
	(0) circle (1.5pt)
	(11) circle (1.5pt)
	(31) circle (1.5pt)
	(32) circle (1.5pt)
	(1111) circle (1.5pt)
	(1121) circle (1.5pt)
;
\draw[fill=white]
	(1) circle (1.5pt)
	(2) circle (1.5pt)
	(3) circle (1.5pt)
	(4) circle (1.5pt)
	(111) circle (1.5pt)
	(112) circle (1.5pt)
	(321) circle (1.5pt)
;
%
\draw
	(11) node[left] {$1$}
	(1111) node[right] {$2$}
	(11)++(0,.1) node[above] {$3$}
	(1121) node[right] {$4$}
	(11) node[right] {$5$}
	(0)++(-.15,.2) node[above left] {$6$}
	(0)++(0,.2) node[above] {$7$}
	(31) node[right] {$8$}
	(32) node[left] {$9$}
	(32) node[right] {$10$}
	(0)++(.1,.2) node[above right] {$11$}
	(0) node[below] {$12$}
;
\end{tikzpicture}
\qquad
%
\begin{tikzpicture}
\coordinate (0) at (0,0);
	\coordinate (1) at (-1.5,1);
		\coordinate (11) at (-1.5,2);
			\coordinate (111) at (-2,3);
				\coordinate (1111) at (-2,4);
			\coordinate (112) at (-1,3);
				\coordinate (1121) at (-1,4);
	\coordinate (2) at (-.5,1);
	\coordinate (3) at (.5,1);
		\coordinate (31) at (0,2);
		\coordinate (32) at (1,2);
			\coordinate (321) at (1,3);
	\coordinate (4) at (1.5,1);
\draw[dotted]
	(0) -- (1)	(0) -- (2)	(0) -- (3)	(0) -- (4)
	(1) -- (11)
	(11) -- (111) -- (1111)	(11) -- (112) -- (1121)
	(3) -- (31)	(3) -- (32) -- (321)
;
\draw
	(1) -- (2) -- (3) -- (4) to[bend left] (0)
	(1) to[bend left] (111) -- (112) to[bend left] (11)
	(111) to[bend left] (1111)
	(112) to[bend left] (1121)
	(3) to[bend left] (31)
	(3) to[bend left] (321) to[bend left] (32)
;
\draw[fill=black]
	(0) circle (1.5pt)
	(11) circle (1.5pt)
	(31) circle (1.5pt)
	(32) circle (1.5pt)
	(1111) circle (1.5pt)
	(1121) circle (1.5pt)
;
\draw[fill=white]
	(1) circle (1.5pt)
	(2) circle (1.5pt)
	(3) circle (1.5pt)
	(4) circle (1.5pt)
	(111) circle (1.5pt)
	(112) circle (1.5pt)
	(321) circle (1.5pt)
;
%
\draw
	(1) node[left] {$0$}
	(111) node[left] {$1$}
	(1111) node[right] {$2$}
	(112) node[right] {$3$}
	(1121) node[right] {$4$}
	(11) node[right] {$5$}
	(2) node[above] {$6$}
	(3) node[above right] {$7$}
	(31) node[left] {$8$}
	(32) node[right] {$9$}
	(321) node[right] {$10$}
	(4) node[above right] {$11$}
	(0) node[below] {$12$}
;
\end{tikzpicture}
\qquad
%
\begin{tikzpicture}
\coordinate (0) at (0,0);
	\coordinate (1) at (-1.5,1);
		\coordinate (11) at (-2,2);
		\coordinate (12) at (-1,2);
			\coordinate (121) at (-1.5,3);
			\coordinate (122) at (-.5,3);
	\coordinate (2) at (1.5,1);
		\coordinate (21) at (1.5,2);
			\coordinate (211) at (.5,3);
			\coordinate (212) at (1.5,3);
				\coordinate (2121) at (1.5,4);
			\coordinate (213) at (2.5,3);
				\coordinate (2131) at (2.5,4);
\draw
	(0) -- (1)
	(0) -- (2)
	(1) -- (11)	(1) -- (12)
	(12) -- (121)	(12) -- (122)
	(2) -- (21) -- (211)
	(21) -- (212) -- (2121)
	(21) -- (213) -- (2131)
;
\draw[fill=black]
	(11) circle (1.5pt)
	(121) circle (1.5pt)
	(122) circle (1.5pt)
	(211) circle (1.5pt)
	(2121) circle (1.5pt)
	(2131) circle (1.5pt)
;
\draw[fill=white]
	(0) circle (1.5pt)
	(1) circle (1.5pt)
	(2) circle (1.5pt)
	(12) circle (1.5pt)
	(21) circle (1.5pt)
	(212) circle (1.5pt)
	(213) circle (1.5pt)
;
%
\draw
	(0) node[below] {$0$}
	(1) node[left] {$1$}
	(11) node[left] {$2$}
	(12) node[left] {$3$}
	(121) node[left] {$4$}
	(122) node[left] {$5$}
	(2) node[right] {$6$}
	(21) node[right] {$7$}
	(211) node[left] {$8$}
	(212) node[left] {$9$}
	(2121) node[left] {$10$}
	(213) node[right] {$11$}
	(2131) node[right] {$12$}
;
\end{tikzpicture}
\end{scriptsize}
\caption{The tree $T^\bullet$ associated with the partition from Fig. \ref{fig:ex_part}, black corners indexed according to the contour sequence, and its image $\mathcal{T}^\bullet$ by the Janson--Stef{\'a}nsson bijection, with its vertices indexed in lexicographical order.}
\label{fig:arbre_non_GW}
\end{figure}
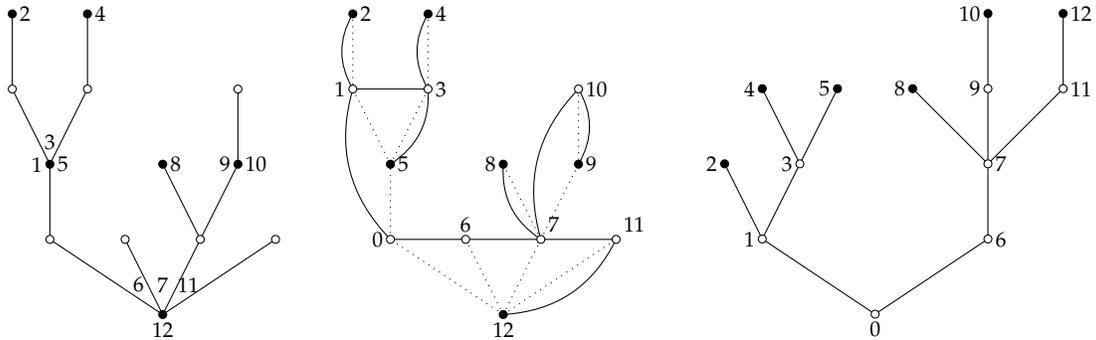

We let
\begin{equation*}
\mathcal{T}^\circ(P) \coloneqq \mathcal{G}(T^\circ(P)) \qquad \textrm{and} \qquad \mathcal{T}^\bullet(P) \coloneqq \mathcal{G}(T^\bullet(P))
\end{equation*}
be the (one-type) trees associated with $T^\circ(P)$ and $T^\bullet(P)$ respectively and now explain how to reconstruct the non-crossing partition $P$ from the trees $\mathcal{T}^\circ(P)$ and $\mathcal{T}^\bullet(P)$. 

To this end, we introduce the notion of \emph{twig}. If $T$ is a tree and $u, v \in T$, denote by $ \llbracket u, v \rrbracket$ the shortest path between $u$ and $v$ in $T$. A twig of $T$ is a set of the form $\llbracket u, v \rrbracket$, where $u$ is an ancestor of $v$ and such that all the vertices of $\rrbracket u, v \rrbracket$ are the last child of their parent; we agree that $\llbracket u, u \rrbracket$ is a twig for every vertex $u$. Now, if $\tau \in \T_{n+1}$ is a tree, let $\varnothing = u(0) \prec u(1) \prec \dots \prec u(n)$ be its vertices listed in lexicographical order. We define two partitions $P_\circ(\tau)$ and $P_\bullet(\tau)$ of $[n]$ as follows:
\begin{itemize}
\item $i, j \in [n]$ belong to the same block of $P_\circ(\tau)$ when $u(i)$ and $u(j)$ have the same parent in $\tau$;
\item $i, j \in [n]$ belong to the same block of $P_\bullet(\tau)$ when $u(i)$ and $u(j)$ belong to a same twig.
\end{itemize}
It is an easy exercise to check that for every $\tau \in \T$, $P_\circ(\tau)$ and $P_\bullet(\tau)$ are indeed partitions which, further, are non-crossing. As illustrated by Fig. \ref{fig:arbre_GW} and Fig. \ref{fig:arbre_non_GW}, we have the following result.

\begin{prop}\label{prop:bijection_pnc_arbres}
For every non-crossing partition $P$ we have
\begin{equation*}
P = P_\circ(\mathcal{T}^\circ(P)) = P_\bullet(\mathcal{T}^\bullet(P)).
\end{equation*}
\end{prop}

\begin{proof}
Fix a non-crossing partition $P$ of $[n]$. Let us first prove the first equality. Define the contour sequence $(u_0, u_1, \dots ,u_{2n})$ of the tree $T^\circ(P)$ as follows: $u_0 = \varnothing$ and for each $i \in \{0, \dots, 2n-1\}$, $u_{i+1}$ is either the first child of $u_i$ which does not appear in the sequence $(u_0, \dots, u_i)$, or the parent of $u_i$ if all its children already appear in this sequence. Recall that a corner of a vertex $v \in T^\circ(P)$ is a sector around $v$ delimited by two consecutive edges. We index from $1$ to $n$ the corners of the black vertices of $T^\circ(P)$, following the contour sequence. By construction of $T^\circ(P)$, we recover $P$ from these corners: for each black vertex of $T^\circ(P)$, the indices of its corners, listed in clockwise order, form a block of $P$. Now assign labels to the vertices of $\mathcal{T}^\circ(P)$ as follows. By definition of the bijection $\mathcal{G}$, each edge of $\mathcal{T}^\circ(P)$ starts from one of these corners, we then label its other extremity by the label of the corner. The root of $\mathcal{T}^\circ(P)$ is not labelled, we assign it the label $0$; the labels thus obtained correspond to the lexicographical order in $\mathcal{T}^\circ(P)$ and the first identity follows.

For the second equality, define similarly the contour sequence of $T^\bullet(P)$, but starting from the first child of the root, and label the black corners as before. We then label the vertices  of $T^\bullet(P)$ as follows: the label of every black vertex is the largest label of its adjacent corners, and then assign the remaining labels of its adjacent corners in decreasing order to its children,  starting from the last one. Observe that the root of $T^\bullet(P)$ has as many children as corners, and all the other black vertices have one child less than the number of corners. Thus all the vertices of $T^\bullet(P)$ have labels, except the first child of the root which we label $0$. We recover $P$ from $T^\bullet(P)$ as follows: for each black vertex of $T^\bullet(P)$, its label, together with the labels of its children, form a block of $P$ (and one does not take into account the label $0$). As the vertex set of $T^\bullet(P)$ and of $\mathcal{T}^\bullet(P)$ is the same, we also get a labeling of the vertices of $\mathcal{T}^\bullet(P)$. Again, by definition of the $\mathcal{G}$, these labels correspond to the lexicographical order in $\mathcal{T}^\bullet(P)$ and the second identity follows.
\end{proof}

Observe from the previous results that the plane trees $\mathcal{T}^\circ(P)$ and $\mathcal{T}^\bullet(P)$ are in bijection. Let us describe a direct operation on trees which maps $\mathcal{T}^\circ(P)$ onto $\mathcal{T}^\bullet(P)$. Starting from a tree $\tau \in \T$, we construct a tree $\mathcal{B}(\tau)$ on the same vertex-set by defining edges (called ``new'' edges in the sequel) as follows: first, we link any two consecutive children in $\tau$; second, we link every vertex $v$ which is the first child of its parent to its youngest ancestor $u$ such that $\llbracket u, pr(v) \rrbracket$ is a twig in $\tau$ (in this case observe that either $u$ is the root of $\tau$, or $v$ is not the last child of $u$ in $ \mathcal{B}(\tau)$).

We leave it as an exercise to check that this mapping preserves the lexicographical order.

\begin{figure}[h!] \centering
\begin{scriptsize}
\begin{tikzpicture}
\coordinate (0) at (0,0);
	\coordinate (1) at (-1.5,1);
		\coordinate (11) at (-1.5,2);
	\coordinate (2) at (0,1);
		\coordinate (21) at (0,2);
	\coordinate (3) at (1.5,1);
		\coordinate (31) at (.5,2);
		\coordinate (32) at (1.25,2);
			\coordinate (321) at (1.25,3);
				\coordinate (3211) at (.75,4);
				\coordinate (3212) at (1.75,4);
		\coordinate (33) at (1.75,2);
		\coordinate (34) at (2.5,2);
\draw
	(0) -- (1) -- (11)
	(0) -- (2) -- (21)
	(0) -- (3)
	(3) -- (31)	(3) -- (32)	(3) -- (33)	(3) -- (34)
	(32) -- (321) -- (3211)	(321) -- (3212)
;
\draw[fill=black]
	(0) circle (1pt)
	(1) circle (1pt)
	(2) circle (1pt)
	(3) circle (1pt)
	(32) circle (1pt)
	(321) circle (1pt)
	(11) circle (1pt)
	(21) circle (1pt)
	(31) circle (1pt)
	(33) circle (1pt)
	(34) circle (1pt)
	(3211) circle (1pt)
	(3212) circle (1pt)
;
%
\draw
	(0) node[below] {$0$}
	(1) node[left] {$1$}
	(11) node[left] {$2$}
	(2) node[left] {$3$}
	(21) node[left] {$4$}
	(3) node[right] {$5$}
	(31) node[right] {$6$}
	(32) node[right] {$7$}
	(321) node[left] {$8$}
	(3211) node[left] {$9$}
	(3212) node[left] {$10$}
	(33) node[right] {$11$}
	(34) node[right] {$12$}
;
\end{tikzpicture}
\quad
%
\begin{tikzpicture}
\coordinate (0) at (0,0);
	\coordinate (1) at (-1.5,1);
		\coordinate (11) at (-1.5,2);
	\coordinate (2) at (0,1);
		\coordinate (21) at (0,2);
	\coordinate (3) at (1.5,1);
		\coordinate (31) at (.5,2);
		\coordinate (32) at (1.25,2);
			\coordinate (321) at (1.25,3);
				\coordinate (3211) at (.75,4);
				\coordinate (3212) at (1.75,4);
		\coordinate (33) at (1.75,2);
		\coordinate (34) at (2.5,2);
\coordinate (3') at (2,1);
\coordinate (321') at (1.75,3);
\draw[dotted]
	(0) -- (1) -- (11)
	(0) -- (2) -- (21)
	(0) -- (3)
	(3) -- (31)	(3) -- (32)	(3) -- (33)	(3) -- (34)
	(32) -- (321) -- (3211)	(321) -- (3212)
;
\draw
	(1) -- (2) -- (3)
	(31) -- (32) -- (33) -- (34)
	(3211) -- (3212)
	(1) to[bend right] (0)
	(11) to[bend left] (1)
	(21) to[bend left] (2)
	(321) to[bend left] (32)
	(3211) to[out=-15,in=90] (321') to[out=-90,in=15] (32)
	(31) to[out=-15,in=90] (3') to[out=-90,in=15] (0)
;
\draw[fill=black]
	(0) circle (1pt)
	(1) circle (1pt)
	(2) circle (1pt)
	(3) circle (1pt)
	(32) circle (1pt)
	(321) circle (1pt)
	(11) circle (1pt)
	(21) circle (1pt)
	(31) circle (1pt)
	(33) circle (1pt)
	(34) circle (1pt)
	(3211) circle (1pt)
	(3212) circle (1pt)
;
%
\draw
	(0) node[below] {$0$}
	(1) node[left] {$1$}
	(11) node[left] {$2$}
	(2) node[below left] {$3$}
	(21) node[left] {$4$}
	(3) node[right] {$5$}
	(31) node[above] {$6$}
	(32) node[above left] {$7$}
	(321) node[left] {$8$}
	(3211) node[left] {$9$}
	(3212) node[right] {$10$}
	(33) node[above right] {$11$}
	(34) node[above right] {$12$}
;
\end{tikzpicture}
\quad
%
\begin{tikzpicture}
\coordinate (0) at (0,0);
	\coordinate (1) at (-1.5,1);
		\coordinate (11) at (-2,2);
		\coordinate (12) at (-1,2);
			\coordinate (121) at (-1.5,3);
			\coordinate (122) at (-.5,3);
	\coordinate (2) at (1.5,1);
		\coordinate (21) at (1.5,2);
			\coordinate (211) at (.5,3);
			\coordinate (212) at (1.5,3);
				\coordinate (2121) at (1.5,4);
			\coordinate (213) at (2.5,3);
				\coordinate (2131) at (2.5,4);
\draw
	(0) -- (1)
	(0) -- (2)
	(1) -- (11)	(1) -- (12)
	(12) -- (121)	(12) -- (122)
	(2) -- (21) -- (211)
	(21) -- (212) -- (2121)
	(21) -- (213) -- (2131)
;
\draw[fill=black]
	(11) circle (1pt)
	(121) circle (1pt)
	(122) circle (1pt)
	(211) circle (1pt)
	(2121) circle (1pt)
	(2131) circle (1pt)
	(1) circle (1pt)
	(2) circle (1pt)
	(12) circle (1pt)
	(21) circle (1pt)
	(212) circle (1pt)
	(213) circle (1pt)
;
%
\draw
	(0) node[below] {$0$}
	(1) node[left] {$1$}
	(11) node[left] {$2$}
	(12) node[left] {$3$}
	(121) node[left] {$4$}
	(122) node[left] {$5$}
	(2) node[right] {$6$}
	(21) node[right] {$7$}
	(211) node[left] {$8$}
	(212) node[left] {$9$}
	(2121) node[left] {$10$}
	(213) node[right] {$11$}
	(2131) node[right] {$12$}
;
\end{tikzpicture}
\end{scriptsize}
\caption{The transformation $\tau \mapsto \mathcal{B}(\tau)$.}
\label{fig:transformation_B}
\end{figure}
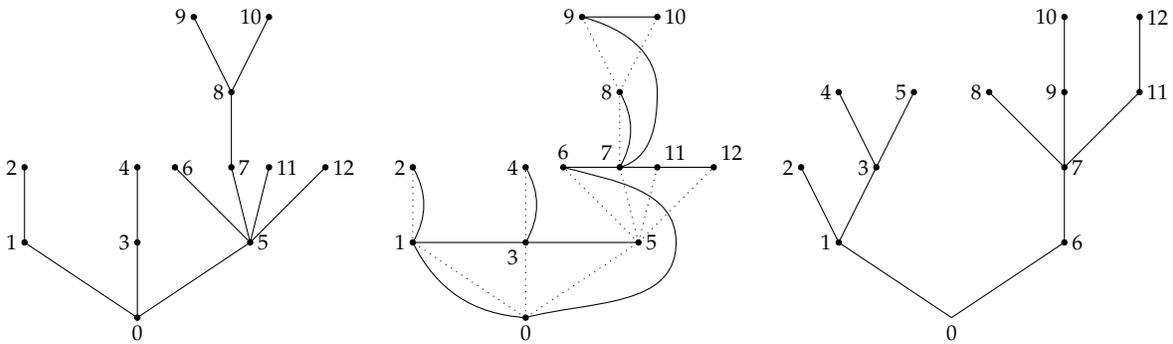

\begin{prop}\label{prop:bijection_directe_arbres}
For every non-crossing partition $P$ we have
\begin{equation*}
\mathcal{B}(\mathcal{T}^\circ(P)) = \mathcal{T}^\bullet(P).
\end{equation*}
\end{prop}

\begin{proof}
Fix a non-crossing partition $P$. Thanks to Prop.~\ref{prop:bijection_pnc_arbres}, it is equivalent to show that
\begin{equation*}
P_\bullet(\mathcal{B}(\mathcal{T}^\circ(P))) = P,
\end{equation*}
and we set by $P'=P_\bullet(\mathcal{B}(\mathcal{T}^\circ(P))) $ to simplify notation.

Suppose first that $i, j \ge 1$ lie in the same block of $P$. We shall show that $i$ and $j$ belong to the same block of $P'$. The two corresponding vertices, say, $u(i)$ and $u(j)$ have the same parent in $\mathcal{T}^\circ(P)$. Without loss of generality, assume that $u(i) \prec u(j)$ are consecutive children  in $\mathcal{T}^\circ(P)$. It suffices to check that, in  $\mathcal{B}(\mathcal{T}^\circ(P))$, $u(j)$ is the last child of $u(i)$. This simply follows from the fact $ \mathcal{B}$ preserves the lexicographical order and that the children of $u(i)$ in $\mathcal{B}(\mathcal{T}^\circ(P))$, $u(j)$ excluded, are descendants of $u(i)$ in $\mathcal{T}^\circ(P)$.

Conversely, suppose that $i, j \ge 1$ lie in the same block of $P'$. Without loss of generality, we may assume that, in $\mathcal{B}(\mathcal{T}^\circ(P))$, $u(j)$ is the last child of $u(i)$. We argue by contradiction and assume that, in $\mathcal{T}^\circ(P)$, $u(i)$ and $u(j)$ are not siblings. We saw that in this case, by definition of $ \mathcal{B}$,  either $u(i)$ is the root, or $u(j)$ is not the last child of $u(i)$ in $\mathcal{B}(\mathcal{T}^\circ(P))$. Both of these cases are excluded. Therefore  $i$ and $j$ belong to the same block of $P$.
\end{proof}

We already mentioned in the Introduction that the bijection $\tau \leftrightarrow P_\circ(\tau)$ was defined by Dershowitz and Zaks \cite{DZ86}; the bijection $\tau \leftrightarrow P_\bullet(\tau)$ was defined by Prodinger \cite{Pro83} and further used in combinatorics, see e.g. Yano and Yoshida \cite{YY07} and in (free) probability, see Ortmann \cite{Ort12}. Roughly speaking, here we unify these two bijections by seeing that they amount (up to the Janson--Stef{\'a}nsson bijection) to choosing different distinguished corners in the dual two-type planar tree. In this spirit, if $P$ is a non-crossing partition, let us also mention that its Kreweras complement $K(P)$ is just obtained by re-rooting $T(P)$ at a new corner; more precisely, the mappings $\left( T^{\bullet} \right) ^{-1} \circ T^{\circ}$ and $\left( \mathcal{T}^{\bullet} \right) ^{-1} \circ \mathcal{T}^{\circ}$ coincide and both correspond to $K$.

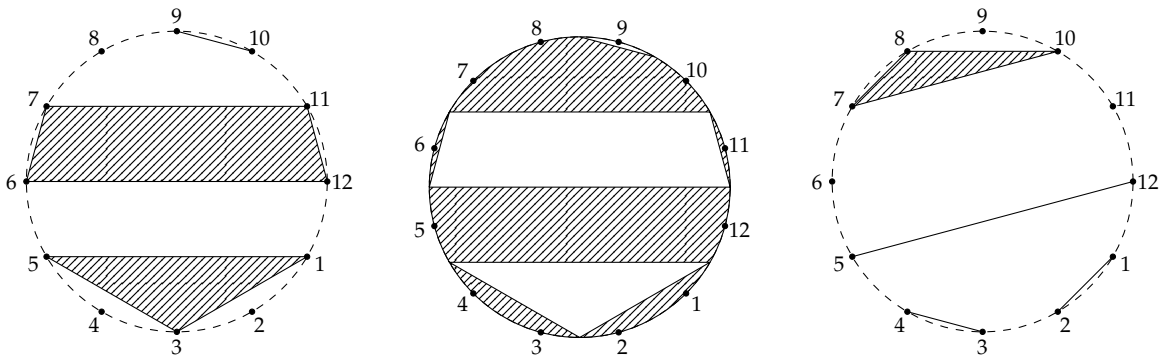
\begin{figure}[ht] \centering
\begin{scriptsize}
\begin{tikzpicture}
\draw[thin, dashed]	(0,0) circle (2);
\foreach \x in {1, 2, ..., 12}
	\coordinate (\x) at (-\x*360/12 : 2);
\foreach \x in {1, 2, ..., 12}
	\draw
	[fill=black]	(\x) circle (1pt)
	(-\x*360/12 : 2*1.1) node {\x}
;
\filldraw[pattern=north east lines]
	(1) -- (3) -- (5) -- cycle
	(6) -- (7) -- (11) -- (12) -- cycle
	(9) -- (10)
;
\end{tikzpicture}
\qquad
%
\begin{tikzpicture}
\draw[thin, dashed]	(0,0) circle (2);
\foreach \x in {1, 2, ..., 12}
	\coordinate (\x) at (-\x*360/12 : 2);
\foreach \x in {1, 2, ..., 12}
	\draw
	[fill=black]	(-360/24-\x*360/12 : 2) circle (1pt)
	(-360/24-\x*360/12 : 2*1.1) node {\x}
;
\filldraw[pattern=north east lines]
	(0:2) arc (0:-360/12:2) -- (-5*360/12:2) arc (-5*360/12:-6*360/12:2) -- cycle
	(-360/12:2) arc (-360/12:-3*360/12:2) -- (1)
	(-3*360/12:2) arc (-3*360/12:-5*360/12:2) -- (3)
	(-6*360/12:2) arc (-6*360/12:-7*360/12:2) -- (6)
	(-7*360/12:2) arc (-7*360/12:-9*360/12:2) -- (-10*360/12:2) arc (-10*360/12:-11*360/12:2) -- cycle
	(-11*360/12:2) arc (-11*360/12:-12*360/12:2) -- (11)
;
\draw 	(-9*360/12:2) arc (-9*360/12:-10*360/12:2);
\end{tikzpicture}
\qquad
%
\begin{tikzpicture}
\draw[thin, dashed]	(0,0) circle (2);
\foreach \x in {1, 2, ..., 12}
	\coordinate (\x) at (-\x*360/12 : 2);
\foreach \x in {1, 2, ..., 12}
	\draw
	[fill=black]	(-\x*360/12 : 2) circle (1pt)
	(-\x*360/12 : 2*1.1) node {\x}
;
\filldraw[pattern=north east lines]
	(1) -- (2)
	(3) -- (4)
	(5) -- (12)
	(7) -- (8) -- (10) -- cycle
;
\end{tikzpicture}
\end{scriptsize}
\caption{The Kreweras complement of the partition $\{\{1, 3, 5\}, \{2\}, \{4\}, \{6, 7, 11, 12\}, \{8\}, \{9, 10\}\}$ is $\{\{1, 2\}, \{3, 4\}, \{5, 12\}, \{6\}, \{7, 8, 10\}, \{11\}\}$.}
\label{fig:complement_Kreweras}
\end{figure}

The Kreweras complement can be formally defined as follows. If we denote by $\NC(A)$ the set of non-crossing partitions on a finite subset $A \subset \N$, then we have canonical isomorphisms $\NC_n \coloneqq \NC(\{1, 2, \dots, n\}) \cong \NC(\{1, 3, \dots, 2n-1\}) \cong \NC(\{2, 4, \dots, 2n\})$. Given two non-crossing partitions $P \in \NC(\{1, 3, \dots, 2n-1\})$ and $P' \in \NC(\{2, 4, \dots, 2n\})$, one constructs a (possibly crossing) partition $P \cup P'$ of $\{1, 2, \dots, 2n\}$. The Kreweras complement of a non-crossing partition $P \in \NC_n \cong \NC(\{1, 3, \dots, 2n-1\})$ is then given by
\begin{equation*}
K(P) = \max\{P' \in \NC_n \cong \NC(\{2, 4, \dots, 2n\}) : P \cup P' \in \NC_{2n}\},
\end{equation*}
where the maximum refers to the partial order of reverse refinement: $P_1 \preceq P_2$ when every block of $P_1$ is contained in a block of $P_2$.

The Kreweras complementation can be visualized as follows: consider the representation of $P \in \NC_n$ in the unit disk as in Fig. \ref{fig:ex_part}; invert the colors and rotate the vertices of the regular $n$-gon by an angle $-\pi/n$; then the blocks of $K(P)$ are given by the vertices lying in the same ``coloured'' component. See Fig. \ref{fig:complement_Kreweras} for an illustration.

\subsection{Simply generated non-crossing partitions and simply generated trees}
\label {sec:sim}

An important feature of the bijection $\mathcal{B}^{\circ} : P \mapsto \mathcal{T}^\circ(P)$ is that it transforms simply generated non-crossing partitions into simply generated trees, which were introduced by Meir \& Moon \cite {MM78} and whose definition we now recall.

Given a sequence $w=(w(i) ; i \geq 0)$ of nonnegative real numbers,  with every $  \tau \in \T$, associate a weight $\Omega^w(\tau)$:
$$ \Omega^w(\tau)= \prod_{u \in \tau} w(k_{u}).$$
Then, for every $ \tau \in \T_{n}$, set
$$ \Q^{w}_{n}(\tau)= \frac{ \Omega^w(\tau)}{ \sum_{T \in \T_{n}} \Omega^w(T)}.$$
Again, we always restrict our attention to those values of $n$ for which $ \sum_{T \in \T_{n}} \Omega^w(T)>0$. A random tree of $ \T_{n}$ sampled according to $ \Q^ {w}_{n}$ is called a simply generated tree. A particular case of such trees on which we shall focus in Sec.~\ref{sec:Hausdorff} is when the sequence of weights $w$ defines a probability measure on $\Z_+$ with mean $1$ (see the discussion in Sec. \ref{sec:sgn} below). In this case, $\Q^w_n$ is the law of a Galton--Watson tree with critical offspring distribution $w$ conditioned to have $n$ vertices.

\begin {prop}\label {prop:bij}Let $ (w(i) ; i \geq 1 )$ be any sequence of nonnegative real numbers. Set $ w(0)=1$. Then, for every $P \in \NC_{n}$,
$$\P^ {w}_{n}(P)=\Q^ {w}_{n+1}(\mathcal{T}^\circ(P)).$$
\end {prop}
In other words, the bijection $\mathcal{B}^{\circ}$ transforms simply generated non-crossing partitions into simply generated trees. 

\begin{proof}
By Prop.~\ref{prop:bijection_pnc_arbres}, we have $P = P_\circ(\mathcal{T}^\circ(P))$. In particular, blocks of size $k \geq1$ in $P$ are in bijection with vertices with out-degree $k$ in $\mathcal{T}^\circ(P)$. The claim immediately follows. 
\end{proof}

It is also possible to give an explicit description of the law of $T^{\circ}$ under $\P^ {w}_{n}$, which turns out to be a two-type simply generated tree. We denote by $\T^{(\mathsf{e}, \mathsf{o})}$ the set of finite two-type trees: for every $\tau \in \T^{(\mathsf{e}, \mathsf{o})}$, we denote by $\mathsf{e}(\tau)$ and $\mathsf{o}(\tau)$ the set of vertices respectively at even and odd generation in $\tau$. Given two sequences of weights $w^\mathsf{e}$ and $w^\mathsf{o}$, we define the weight of tree $\tau \in \T^{(\mathsf{e}, \mathsf{o})}$ by
\begin{equation*}
\Omega^{(w^\mathsf{e}, w^\mathsf{o})}(\tau) = \prod_{u \in \mathsf{e}(\tau)} w^\mathsf{e}(k_{u}) \prod_{u \in \mathsf{o}(\tau)} w^\mathsf{o}(k_{u}).
\end{equation*}
and we define for every $\tau \in \T^{(\mathsf{e}, \mathsf{o})}_n$ the set of two-type trees with $n$ vertices,
\begin{equation*}
\Q^{(w^\mathsf{e}, w^\mathsf{o})}_n(\tau) = \frac{\Omega^{(w^\mathsf{e}, w^\mathsf{o})}(\tau)}{\sum_{T \in \T^{(\mathsf{e}, \mathsf{o})}_n} \Omega^{(w^\mathsf{e}, w^\mathsf{o})}(T)},
\end{equation*}
where, again, we implicitly restrict ourselves to the values of $n$ for which $\sum_{T \in \T^{(\mathsf{e}, \mathsf{o})}_{n}} \Omega^{(w^\mathsf{e}, w^\mathsf{o})}(T)>0$. A random tree sampled according to $\Q^{(w^\mathsf{e}, w^\mathsf{o})}_n$ is called a two-type simply generated tree.

\begin{prop}\label{prop:PNC_et_arbre_deux_types}
Let $w = (w(i), i \ge 1)$ be a sequence of nonnegative real numbers and $c > 0$ be a positive real number. For every $i \ge 0$, set $w^\mathsf{o}(i) = w(i+1)$ and $w^\mathsf{e}(i) = c^{-(i+1)}$. Then, for every $P \in \NC_n$,
\begin{equation*}
\P^w_n(P) = \Q^{(w^\mathsf{e}, w^\mathsf{o})}_{n+1}(T^\circ(P)).
\end{equation*}
\end{prop}

\begin{proof}
Fix $P \in \NC_n$; by construction of $T^\circ(P)$ (recall the proof of Prop.~\ref{prop:bijection_pnc_arbres}), the vertices at odd generation in $T^\circ(P)$ are in bijection with the blocks of $P$ and the degree of each corresponds to the size of the associated block. Consequently, we have on the one hand
\begin{equation*}
\prod_{u \in \mathsf{o}(T^\circ(P))} w^\mathsf{o}(k_{u})
= \prod_{u \in \mathsf{o}(T^\circ(P))} w(k_{u}+1) 
= \prod_{B \textrm{ block of } P} w(\textrm{size of }B)
= \Omega^w(P);
\end{equation*}
on the other hand, since $T^\circ(P) \in \T^{(\mathsf{e}, \mathsf{o})}_{n+1}$,
\begin{equation*}
\prod_{u \in \mathsf{e}(T^\circ(P))} w^\mathsf{e}(k_{u})
= \prod_{u \in \mathsf{e}(T^\circ(P))} c^{-(k_{u}+1)}
= c^{-\sum_{u \in \mathsf{e}(T^\circ(P))} (k_{u}+1)}
= c^{-(n+1)}.
\end{equation*}
This last term only depends on $n$ and not on $P$ and the claim follows.
\end{proof}

Recall that $\mathcal{G}$ denotes the Janson--Stef{\'a}nsson bijection. Then, combining Propositions \ref{prop:bij} and \ref{prop:PNC_et_arbre_deux_types}, we obtain the following result.

\begin{cor}\label{cor:bijection_JS_arbres_simplement_generes}
Let $w = (w(i), i \ge 1)$ be any sequence of nonnegative real numbers and $c > 0$ be a positive real number. Set $w(0)=1$ and for every $i \ge 0$, define $w^\mathsf{o}(i) = w(i+1)$ and $w^\mathsf{e}(i) = c^{-(i+1)}$. Then, for every $T \in \T^{(\mathsf{e}, \mathsf{o})}_n$, we have
\begin{equation*}
\Q^{w^\mathsf{e}, w^\mathsf{o}}_n(T) = \Q^w_n(\mathcal{G}(T)).
\end{equation*}
\end{cor}

In other words, the Janson--Stef{\'a}nsson bijection transforms a certain class of two-type simply generated trees into one-type simply generated trees. A similar result implicitly appears in their work \cite[Appendix A]{JS12} in the particular case of Galton--Watson trees, where $w^\mathsf{e}$ and $w^\mathsf{o}$ are probability distribution on $\{0, 1, \dots,\}$ and moreover $w^\mathsf{e}$ is a geometric distribution.

\section{Applications}

In this section, we use  simply generated trees to study combinatorial properties of simply generated non-crossing partitions.  Indeed, as suggested by Prop.~\ref{prop:bij}, it is possible to reformulate questions concerning random non-crossing partitions in terms of random trees, which are more familiar grounds.

\subsection {Asymptotics of simply generated trees}
\label {sec:sgn}
Following Janson \cite {Jan12}, here we describe all the possible regimes arising in the asymptotic behavior of simply generated trees. All the following discussion appears in  \cite {Jan12}, but we reproduce it here for the reader's convenience in view of future use and refer to the latter reference for details and proofs.

Let $(w(i); i \geq 0)$ be a sequence of nonnegative real numbers with $ w(0)>0$ and $ w(k)>0$ for some $k \geq 2$ (and keeping in mind that we will take $w(0)=1$ in view of Prop.~\ref {prop:bij}). Set
$$ \Phi(z)= \sum_{k=0}^{\infty} w(k) z^{k} ,  \qquad  \Psi(z)= \frac{z \Phi'(z)}{\Phi(z)}=  \frac{ \sum_{k=0}^{\infty} k w(k) z^{k}}{ \sum_{k=0}^{ \infty} w(k) z^{k}}, \qquad \rho = \left(  \limsup_{k \rightarrow \infty} w(k)^{1/k} \right)^{-1}.$$
If $ \rho=0$, set $ \nu=0$ and otherwise $$\nu = \lim_{t \uparrow \rho}  \Psi(t).$$
We now define a number $ \xi \geq 0$ according to the value of $ \nu$.
\begin{itemize}

\item If $\nu \geq 1$, then $ \xi$ is the unique number in $(0,\rho]$ such that $ \Psi(\xi)=1$. 
\item If $ \nu<1$, then we set $ \xi= \rho$.
\end{itemize}
In both cases, we have $ 0< \Phi(\xi)< \infty$, and we set
$$ \pi(k)= \frac{ w(k) \xi^{k}}{\Phi(\xi)}, \qquad k \geq 0,$$
so that $ \pi$ is a probability distribution with expectation $\min(\nu,1)$ and variance $\xi \Psi'(\xi) \le \infty$. 

We say that another sequence of weights $\widetilde{w} = (\widetilde{w}(i); i \ge 0)$ is equivalent to $w$ when there exists $a, b > 0$ such that $\widetilde{w}(i) = a b^{i} w(i)$ for every $i \ge 0$. In this case, one can check that $\Omega^{\widetilde{w}}(\tau) = a^n b^{n-1} \Omega^w(\tau)$ for every $\tau \in \T_n$ so that $\Q^{\widetilde{w}}_n = \Q^w_n$ and $\P^{\widetilde{w}}_{n} = \P^{w}_{n}$ for every $n \geq 1$. We see that $w$ is equivalent to a probability distribution if and only if $ \rho>0$. In addition, when $ \rho>0$,  $ \pi$ defined as above is the unique probability distribution with mean $1$ equivalent to $ w$, if such distribution exists; if no such distribution exists, then $ \pi$ is the  probability distribution equivalent to $ w$ that has the maximal mean.

\begin {ex}\label{ex:mesures_proba_equivalentes} Let $ \mathcal{A}$ is a non-empty subset of $\{1,2,3, \ldots\}$ with $ \mathcal{A} \neq \{1\}$. Set $w_{ \mathcal{A} }(0)=1$, $ w_{ \mathcal{A} }(k)=1$ if $k \in \mathcal{A}$ and $ w_{ \mathcal{A} }(k)=0$ if $k \not \in \mathcal{A}$. Then the equivalent probability measure $ \pi_{ \mathcal{A} }$ is defined by 
$$ \pi_{A}(k)= \frac{ \xi_{ \mathcal{A} }^{k}}{ 1+\sum_{i \in \mathcal{A}}  \xi_{ \mathcal{A} }^{i} } \mathbbm {1}_{k \in \{0\} \cup \mathcal{A}}, \qquad \textrm {where }  \xi_{ \mathcal{A} }>0 \textrm { is such that } \qquad  1+\sum_{i \in \mathcal{A} }  \xi_{ \mathcal{A} }^{i}= \sum_{i \in \mathcal{A} } i  \cdot  \xi_{ \mathcal{A} }^{i}.$$
For example,  for fixed $n \geq 1$, we have
$$ {\pi_{n \N}}(k)=  \frac{n}{(1+n)^{1+k/n}} \mathbbm{1}_{k \in n \Z_{+}} \qquad (k \geq 0).$$
In particular, $\pi_{ \Z_{+}}(k)= {1}/{2^{k+1}}$ for every $k \geq 0$. Also,
$$ \pi_{(2 \Z_{+}+1)}(k)= \frac{1-z^{2}}{1+z-z^{2}} \cdot  z^{k}  \cdot \mathbbm {1}_{k=0 \textrm { or } k \textrm{ odd}} \qquad (k \geq 0),$$
where $z$ is the unique root of $1 - 2 z^2 - 2 z^3 + z^4=0$ in $[0,1]$.
\end {ex}

Now let $T_{n}$ be a random element of $ \T_{n}$ sampled according to $ \Q^ {w}_{n}$.

\begin {thm}[Janson \cite{Jan12}, Theorems 7.10 and 7.11] \label{thm:jan1}Fix $k \geq 0$.
\begin{enumerate}
\item We have $ \Pr {k_{\varnothing}(T_{n})=k}  \to k \pi(k)$ as $n \rightarrow \infty$;
\item Let $N_{k}(T_{n})$ be the number of vertices with outdegree $k$ in $T_{n}$. Then $N_{k}(T_{n})/n$ converges in probability to $ \pi_{k}$ as $n \rightarrow \infty$.
\end{enumerate}
\end {thm}

\begin {thm}[Janson \cite{Jan12}, Theorem 18.6]\label{thm:jan2} If $ \rho>0$, we have  
$$   \frac{1}{n}\cdot \log \left(  \sum_{T \in \T_{n}} \Omega^w(T) \right)   \quad \mathop{\longrightarrow}_{n \rightarrow \infty} \quad     \log \left(  {\Phi(\xi)}/{ \xi} \right).$$
\end {thm}
 
\subsection{Applications in the enumeration of non-crossing partitions with prescribed block sizes}
 
By Prop.~\ref{prop:bijection_pnc_arbres}, counting  non-crossing partitions of $[n]$ with conditions on the number of blocks of given sizes reduces to counting plane trees of $\T_{n+1}$ with conditions on the number of vertices with given outdegrees, which is a well-paved road (see e.g.~\cite[Sec.~5.3]{Sta99}). Since our main interest lies in probabilistic aspects of non-crossing partitions, we shall only give one such example of application.  Let $ \mathcal{A}$ be a non-empty subset of $\{1,2,3, \ldots\}$ with $ \mathcal{A} \neq \{1\}$, and denote by $ \NC^{ \mathcal{A} }_{n}$ the set of all non-crossing partitions of $[n]$ with blocks of size only belonging to $ \mathcal{A}$. Recall the definition of $\xi_{\mathcal{A}}$ from Example \ref{ex:mesures_proba_equivalentes}.
  
\begin{prop}\label{prop:asymp}Set $\Phi(z)= 1+ \sum_{k \in \mathcal{A}}z^{k}$. Then
$$ \# \NC^{ \mathcal{A} }_{n}   \quad \mathop{\thicksim}_{n \rightarrow \infty} \quad   \gcd( \mathcal{A}) \cdot  \sqrt{ \frac{\Phi(\xi_{\mathcal{A}})}{2\pi\Phi''(\xi_{\mathcal{A}}) }}  \cdot \left( \frac{\Phi(\xi_{\mathcal{A}})}{\xi_{\mathcal{A}}} \right) ^{n+1}  \cdot n^{-3/2},$$
where $n \rightarrow  \infty$ in such a way that $n$ is divisible by $\gcd( \mathcal{A})$.
\end{prop}

Setting $\overline{\mathcal{A}}= \{0\} \cup \mathcal{A}$, observe that $\#\NC^{ \mathcal{A} }_{n}=\#\T^{ \overline{\mathcal{A}}}_{n+1}$ by Prop.~\ref{prop:bijection_pnc_arbres}. But, by \cite[Prop.~I.5.]{FS09}, the generating function $T^{ \overline{\mathcal{A}}}(z)= \sum_{n \geq 1} \#\T^{ \overline{\mathcal{A}}}_{n} \cdot z^{n}$ satisfies the implicit equation $T^{ \overline{\mathcal{A}}}(z)=z \Phi(T^{ \overline{\mathcal{A}}}(z))$. Prop.~\ref{prop:asymp} then immediately follows from \cite[Thm.VII.2 and Rem.~VI.17]{FS09}.

Let us mention that explicit expressions for $\# \NC^{\mathcal{A}}_n$ for $n$ fixed are known for two particular choices of $\mathcal{A}$.  Edelman \cite{Edel80} has found an explicit formula for $\# \NC^{k \Z_+}_{kn} $  (i.e.~for $k$-divisible non-crossing partitions) and Arizmendi \& Vargas \cite{AV12} have found the explicit expression of $\# \NC^{\{k\}}_{kn} $ (i.e.~for $k$ equal non-crossing partitions):
\begin{equation*}
\# \NC^{\{k\}}_{kn} = \frac{1}{(k-1)n+1} {kn \choose n}
\qquad\text{and}\qquad
\# \NC^{k \Z_+}_{kn} = \frac{1}{kn+1} {(k+1)n \choose n}.
\end{equation*}

\subsection{Applications in free probability}
\label{sec:free}

Recall from the Introduction the definition of the $R$-transform $R_{\mu}$ of a compactly supported probability measure $ \mu$ on the real line, and that it is related to its associated free cumulants $ ( \kappa_{i}(\mu); i \geq 0)$ by the formula
$$ R_{\mu}(z)=\sum_{n=0}^{\infty}\kappa_{n+1}(\mu) z^{n}.$$

\begin{thm}\label {thm:asymp}Let $  \mu$ be a compactly supported probability measure on $ \mathbb{R}$ different from a Dirac mass. Assume that its free cumulants $(\kappa_{i}(\mu); i \geq 1)$ are all nonnegative. Set $$\rho = \left(  \limsup_{n \rightarrow \infty} \kappa_{n}(\mu)^{1/n} \right)^{-1} \qquad \textrm {and}\qquad  \nu =1+ \lim_{t \uparrow \rho}\frac{t^{2} R_{\mu}'(t)-1 }{t R_{\mu}(t)+1}.$$
\begin{enumerate}
\item If $ \nu \geq 1$, there exists a unique number $ \xi$ in $(0,\rho]$ such that $R'_{\mu}(\xi)=1/\xi^{2}$ and
$$ \frac{1}{n} \cdot \log \int_{ \mathbb{R}} t^{n} \mu(dt)  \quad\mathop{\longrightarrow}_{n \rightarrow \infty} \quad    \log \left( \frac{1}{\xi}+  R_{\mu}(\xi)  \right) .$$
\item If $ \nu<1$, we have
$$ \frac{1}{n} \cdot \log \int_{ \mathbb{R}} t^{n} \mu(dt)  \quad\mathop{\longrightarrow}_{n \rightarrow \infty} \quad    \log \left( \frac{1}{\rho}+  R_{\mu}(\rho)  \right) .$$
\end{enumerate}
\end{thm}

Note that the equality $R'_{\mu}(\xi)=1/\xi^{2}$ is equivalent to $ K_{ \mu}'(\xi)=0$, where we recall that $K_{\mu}$ denotes the inverse of the Cauchy transform of $\mu$.

\begin {proof}First note that $ \rho>0$, as $ R_{\mu}$ is analytic on a neighbourhood of the origin. We  then apply the results of Sec.~\ref {sec:sgn} with weights $w$ defined by $w(0)=1$ and $w(i)= \kappa_{i}(\mu)$ for $i \geq 1$. The fact that $\mu$ is different from a Dirac mass guaranties that $w(k)>0$ for some $k \geq 2$. Observe that
 $$ \Phi(z)=1+zR_{\mu}(z)= z K_{\mu}(z) \qquad\text{and}\qquad \Psi(z)=1+\frac{ z^{2}R'(z)-1}{z R(z)+1}.$$
 In particular, $ \Psi(z)=1$ if and only if $ R_{\mu}'(z)=1/z^{2}$. The claim then follows by combining \eqref{eq:part} with Theorem \ref{thm:jan2}. 
\end {proof}

See Example \ref {ex:nu} below for an example where $ \nu<1$. If $ \mu$ is the uniform measure on $[0,1]$, its free cumulants are not all nonnegative, as $ R_{\mu}(z)={1}/(1-e^{-z})- 1/{z}$. See also \cite {Ben06} for information concerning Taylor series of the $R$-transform of measures which are not compactly supported.

Let $s_{\mu}$ be the maximum of the support of a compactly supported probability measure $ \mu$ on $ \R$. It is well known and simple to check that
$$ \log(s_{\mu})= \limsup_{n \rightarrow \infty} \frac{1}{n} \log \int_{\R} t^{n} \mu(dt).$$
Hence, taking into account \eqref{eq:part},  we immediately get Theorem \ref{thm:support} from Theorem \ref {thm:asymp}.

\begin {ex} \label {ex:nu}
\begin{enumerate}
\item If $ \mu(dx)= 1/( \pi \sqrt {1-x^{2}}) \mathbbm {1}_{|x| \leq 1} dx$ is the arcsine law (which is also the free additive convolution $ \lambda \boxplus \lambda$ with $ \lambda= ( \delta_{-1/2}+ \delta_{1/2})/2)$),  one has $ \rho= \infty$, $ \nu=0$, so that $R_{\mu}(z)= ( \sqrt {1+z^{2}}-1)/z$,
and one recovers that $s_{\mu}=1/ \infty+R(\infty)=1$.

\item If  $ \mu$ is the free convolution of a free Poisson law of parameter $1$ and the uniform distribution on $[-1,1]$, then $R_{\mu}(z)= \mathrm {coth}(z)-z^{-1}+ (1-z)^{-1}$, $ \rho=1$, $\nu= \infty $ so that
$$s_{\mu}= \mathrm {coth}(z_{\ast})+ \frac{1}{1-z_{\ast}} \simeq 4.16 , \qquad \textrm {where } \quad \textrm {csch}(z_{\ast})(1-z_{\ast})^{2}=1 \textrm { with } z_{\ast} \in (0,1).$$
This gives a simpler expression that the one of \cite[Example 6.2]{Ort12}, which involves solutions of two implicit equations.

\item If $ \mu$ is such that $R_{\mu}(z)= \frac{1}{z}- \pi \mathrm {cot}(\pi z)$ (this corresponds to the L\'evy area corresponding to the free Brownian bridge introduced in \cite {Ort13}), then
$$s_{\mu}=  \frac{2- \sqrt {2- \pi^{2} z_{\ast}^{2}}}{z_{\ast}}\simeq 3.94, \qquad \textrm {where } \quad  \frac{ \sin(\pi z_{\ast})}{ \pi z_{\ast} }= \frac{ \sqrt {2}}{2} \textrm { with } z_{\ast} \in (0,1).$$
This gives a simpler expression that the one of \cite[Prop.~ 5.12]{Ort12},

\item As noted by Ortmann \cite[Sec.~6.1]{Ort12}, if $\lambda$ is a finite compactly supported measure on $ \R$ and $ \alpha \in \R$,  by \cite {BV93} or \cite[Thm.~3.3.6]{HP00}, there exists a compactly supported probability measure $ \mu$ such that
$$R_{ \mu}(z)= \alpha+ \int \frac{z}{1-xz} \lambda(dx),$$
and all the cumulants of $ \mu$ are nonnegative, so that  Theorem \ref {thm:asymp} and Theorem \ref {thm:support} apply to the corresponding normalized probability measure. This actually corresponds to the class of so-called freely infinitely divisible measures. 

In particular, if $ \lambda(dx)=c (1-x)^{\alpha}  \mathbbm {1}_{0 \leq x \leq 1} dx$ with $c>0, \alpha>1$, then $\mu$ is such that $ R_{\mu}(z)= \int_{ \R} \frac{z}{1-xz} \lambda(dx)$ and
$$\kappa_{1}(\mu)=0, \qquad  \kappa_{n}(\mu)= c \frac{\Gamma(1+\alpha) \cdot \Gamma(n-1)}{\Gamma(n+\alpha)} \quad  (n \geq 2), \qquad  \rho=1, \qquad  \nu=\frac{(2 \alpha-1) c}{(\alpha-1) (\alpha+c)}.$$
Note that $\kappa_{n}(\mu) \sim c \Gamma(1+ \alpha) \cdot n^{-1-\alpha}$ as $n \rightarrow  \infty$ and that $\nu=1$ if and only if $c=\alpha-1$. For example, for $ \alpha=2$ and $c=1/2$, we have $\nu=3/5<1$ and $s_{\mu}=1+R_{\mu}(1)=5/4$.
\end{enumerate}

\end {ex}

\subsection{Distribution of the block sizes in random non-crossing partitions}
\label{sec:appli}

We are now interested in the distribution of block sizes in large simply generated non-crossing partitions.  We fix a sequence of nonnegative weights $w = (w(i); i \geq 1)$ such that  $w(k)>0$ for some $k \geq 2$.  Set $w(0)=1$, and let $P_{n}$ be a random non-crossing partition with law $\P^w_{n}$. Denote by $ \pi$ the probability distribution equivalent to the weights $w$ in the sense of Sec.~\ref{sec:sgn}. Finally, set $T_{n+1}=\mathcal{T}^\circ(P_{n})$, so that by Prop.~\ref{prop:bij}, $T_{n+1}$ is a simply generated tree with $n+1$ vertices with law $\Q^{w}_{n+1}$.

\paragraph{Blocks of given size.}  If $P$ is a non-crossing partition and $A$ is a non-empty subset of $\N$, we let $ \zeta_{A}(P)$ be the number of blocks of $P$ whose size belongs to $A$. In particular, notice that $\zeta_{\N}(P)$ is the total number of blocks of $P$. 

\begin{thm}\label{thm:blocks1}
\begin{enumerate}
\item[(i)] Let $S_{1}(P_{n})$ be the size of the block containing $1$ in $P_{n}$. Then, for every $k \geq 1$, $ \Pr{S_{1}(P_{n})=k} \rightarrow k \pi(k)$ as $n \rightarrow \infty$.
\item[(ii)] Let $B_{n}$ be a block chosen uniformly at random in $P_{n}$. Assume that $\pi(0)<1$. Then, for every $k \geq 1$, $\Pr{|B_{n}|=k} \rightarrow \pi(k)/(1-\pi(0))$ as $n \rightarrow  \infty$.
\item[(iii)] Let $A$ be a non-empty subset of $\N$. As $n \rightarrow  \infty$, the convergence $ \zeta_{A}(P_{n})/n \rightarrow  \pi(A)$ holds in probability and, in addition, $\Es{\zeta_{A}(P_{n})}/n \rightarrow  \pi(A)$.\end{enumerate}
\end{thm}

In particular, the total number of blocks of $P_{n}$ is of order $(1-\pi(0)) n$ when $\pi(0)<1$.

\begin{proof}
For (i), simply note that $S_{1}(P_{n})=k_{\varnothing}(T_{n+1})$, and the claim immediately follows from Theorem \ref{thm:jan1} (i). For the second assertion, if $T$ is a tree, denote by $N_{k}(T)$ the number of vertices of $T$ with outdegree $k$. Note that $B_{n}$ has the law of the outdegree of an internal (i.e. not a leaf) vertex of $T_{n+1}$ chosen uniformly at random. As a consequence,
$$\Pr{|B_{n}|=k}= \Es{ \frac{N_{k}(T_{n+1})}{n-N_{0}(T_{n+1})}}.$$
By Theorem \ref{thm:jan1} (ii), ${N_{k}(T_{n+1})}/(n-N_{0}(T_{n+1}))$ converges in probability to $ \pi(k)/(1-\pi(0))$ as $k \rightarrow  \infty$, and is clearly bounded by $1$. The second first assertion then follows from the dominated convergence theorem. For the last assertion, observe that $\zeta_{A}(P_{n})=N_{A}(T_{n+1})$, where $N_{A}(T_{n+1})$ denotes the number of vertices of $T_{n+1}$ with outdegree in $A$. Then, fix $K \geq 1$, and to simplify notation, set $A_{K}= A \cap [K]$, so that by Theorem \ref{thm:jan1} (ii), the convergence $ \zeta_{A_{K}}(P_{n})/n \rightarrow  \pi(A_{K})$ holds in probability as $n \rightarrow \infty$. Since $|\zeta_{A}(P_{n})-\zeta_{A_{K}}(P_{n})| \leq n/K$, the quantity $| \zeta_{A}(P_{n})/n- \zeta_{A_{K}}(P_{n})/n|$ can be made arbitrarily small by choosing $K$ sufficiently large. It follows that $ \zeta_{A}(P_{n})/n \rightarrow  \pi(A)$ in probability as $n \rightarrow  \infty$, and the last claim readily by the dominated convergence theorem.
\end{proof}

In the case $\pi(0)=1$  (which corresponds to $\rho=0$), (i) tells us that the convergence $S_{1}(P_{n}) \rightarrow \infty$ holds in probability as $n \rightarrow \infty$, but the asymptotic behavior of $|B_{n}|$  and the total number of blocks of $P_{n}$ remains unclear. Unfortunately, it seems that one cannot say anything more in full generality. Indeed:
\begin{enumerate}
\item[(i)]If $w(k)=k!^{\alpha}$ with $\alpha>1$, by \cite[Remark 2.9]{JJS11}, with probability tending to one as $n \rightarrow \infty$, the root of $T_{n+1}$ has  $n$ children which are all leaves. Therefore, as $n \rightarrow \infty$,  $\Pr{S_{1}(P_{n})=n} \rightarrow 1$,  $\Pr{|B_{n}|=n} \rightarrow 1$ and $\Pr{\zeta_{\N}(P_{n})=1} \rightarrow 1$.
\item[(ii)]   If $w(k)=k!$, by \cite[Theorem 2.4]{JJS11}, with probability tending to one as $n \rightarrow \infty$, the root of $T_{n+1}$ has  $n-U_{n+1}$ children which are all leaves, except $U_{n+1}$ of them (which have only one vertex grafted on them), and $U_{n+1}$ converges in distribution to $X$, a Poisson random variable of parameter $1$, as $n \rightarrow \infty$. Therefore, as $n \rightarrow \infty$, $n-S_{1}(P_{n}) \rightarrow X$  in distribution,   $\Pr{|B_{n}|=1} \rightarrow \Es{X/(X+1)}=1/e$, $\Pr{|B_{n}|=S_{1}(P_{n})} \rightarrow 1-1/e$ and  $\zeta_{\N}(P_{n}) \rightarrow X+1$ in distribution.
\item[(iii)] If $w(k)=k!^{\alpha}$ with $0 <\alpha<1$ and $1/\alpha \not \in \N$ for simplicity, by \cite[Theorem 2.5]{JJS11}, as $n \rightarrow  \infty$, $k_{\varnothing}(T_{n+1})/n \rightarrow 1$ in probability, for every $1 \leq i \leq \lfloor 1/\alpha \rfloor$, $N_{i}(T_{n})/n^{1-i\alpha} \rightarrow i!^{\alpha}$ in probability and, with probability tending to one as $n \rightarrow  \infty$,  $N_{i}(T_{n})=0$ for every $i > \lfloor 1/\alpha \rfloor$. Therefore, as $n \rightarrow \infty$, $S_{1}(P_{n})/n \rightarrow 1$ in probability. Also, noting that
$$\Pr{|B_{n}|=k}= \Es{ \frac{N_{k}(T_{n+1})}{\sum_{i \geq 1}N_{i}(T_{n+1})}}, \qquad  \zeta_{\N}(P_{n})= \sum_{i \geq 1}N_{i}(T_{n+1}),$$
we get that
and  $\Pr{|B_{n}|=1} \rightarrow1$ and  $\zeta_{\N}(P_{n})/n^{1-\alpha} \rightarrow 1$ in probability.
\end{enumerate}
In addition, \cite[Example 19.39]{Jan12} gives an example where $\rho=0$ and $k_{\varnothing}(T_{n})/n \rightarrow 0$ in probability.

\paragraph{Asymptotic normality of the block sizes.} Theorem \ref{thm:blocks1} (ii) shows that a law of large numbers holds for $ \zeta_{A}(P_{n})$. Under some additional regularity assumptions on the weights, it is possible to obtain a central limit theorem. Specifically, assume that $w$ is equivalent (in the sense of Sec.~\ref{sec:sgn}) to a probability distribution $\pi$ which is critical (meaning that its mean is equal to $1$) and has finite positive variance $\sigma^{2}$. In this case, the following result holds.

\begin{thm}\label{thm:normality}Fix an integer $k \geq 1$, and let $A_{1}, \ldots, A_{k}$ be non-empty subsets of $ \N$. Then there exists a centered Gaussian vector $(X_{A_{1}}, \ldots,X_{A_{k}})$ such that the convergence
$$ \left( \frac{ \z_{A_{1}}( P_{n})  - \pi(A_{1}) n}{\sqrt {n}}, \ldots, \frac{ \z_{A_{k}}( {P}_{n})  - \pi(A_{k}) n}{\sqrt {n}} \right)  \quad\mathop{\longrightarrow}^{(d)}_{n \rightarrow \infty} \quad  (X_{A_{1}}, \ldots, X_{A_{k}})$$
holds in distribution. In addition  we have $ \Es {X_{A_{i}}^{2}}= \pi(A_{i})(1- \pi(A_{i}))- \frac{1}{\sigma^{2}} \sum_{r \in A_{i}} (r-1)^{2} \pi(r)$ \quad for $1 \leq i \leq k$ and
$$ \Cov(X_{A_{i}},X_{A_{j}})=- \pi(A_{i}) \pi(A_{j})-\frac{1}{\sigma^{2}} \sum_{r \in A_{i}} (r-1)^{2} \pi(r) \cdot \sum_{s \in A_{j}} (s-1)^{2} \pi(s)$$
if $1 \leq i \neq j \leq k$ are such that  $ A_{i} \cap A_{j} = \emptyset$.
\end{thm}

This result is just a translation of the corresponding known result for conditioned Galton--Watson trees: recalling that $T_{n+1}=\mathcal{T}^\circ(P_{n})$, let  $N_{A}(T_{n+1})$ denote the number of vertices of $T_{n+1}$ with outdegree in $A$, then $(\z_{A_{1}}( P_{n}), \ldots,\z_{A_{k}}( {P}_{n}))=(N_{A_{1}}(T_{n+1}), \ldots, N_{A_{k}}(T_{n+1}))$, and Theorem \ref{thm:normality} then follows from \cite[Example 2.2]{Jan14} (in this reference, the results are stated when $ \# A_{i}=1$ for every $i$, but it is a simple matter to see that they still hold).

\paragraph{Large deviations for the empirical block size distribution.} Denote by $ \mathcal{M} _{n}$ the law of the size of a block of $P_{n}$, chosen uniformly at random among all possible blocks, so that $ \mathcal{M} _{n}$ is a random probability measure on $\N$. Dembo, M\"{o}rters \& Sheffield \cite[Thm.~2.2]{DMS05} establish a large deviation principle for the empirical outdegree distribution in Galton--Watson trees. Therefore, we believe that an analogue large deviation principle holds for  $ \mathcal{M} _{n}$ (at least when the weights are equivalent to a critical probability distribution having a finite exponential moment), which would in particular extend a result of  Ortmann \cite[Thm.~1.1]{Ort12}, who established such a large deviation principle in the case of uniformly distributed $k$-divisible non-crossing partitions. The point is that Ortmann uses the bijection $P \leftrightarrow \mathcal{T}^\bullet(P_{n})$, but we believe that it is simpler to use the bijection $P \leftrightarrow \mathcal{T}^\circ(P)$ since $\mathcal{T}^\circ(P_{n})$ is a simply generated tree, but in general not $\mathcal{T}^\bullet(P_{n})$. However, we have not worked out the details.

\paragraph{Largest blocks.} Depending on the weights, Janson \cite[Sec.~9 and 19]{Jan12} obtains general results concerning the largest outdegrees of simply generated trees. Since the sequence of outdegrees of vertices of $T_{n+1}$ that are not leaves, listed in non increasing order, is equal to the sequence of sizes of blocks of $P_{n}$, listed in non increasing order, one gets estimates on the sizes of the largest blocks of $P_{n}$. We do not enter details, and refer to \cite{Jan12} for precise statements.

\paragraph{Local behavior.}
Theorem \ref{thm:blocks1}~(i) describes the distributional limit of the size of the block of $P_{n}$ containing $1$; it is also possible to describe the behavior of the blocks at ``finite distance'' of the latter. Indeed, as we have seen in Section \ref{sec:sim}, when $P_n$ is sampled according to $\P_n^w$, then its two-type dual tree $T^\circ_n = T^\circ(P_n)$ is distributed according to $\Q^{(w^\mathsf{e}, w^\mathsf{o})}_{n+1}$ where $w^\mathsf{o}(i) = w(i+1)$ and $w^\mathsf{e}(i) = 1$ for every $i \ge 0$. In this case, for every tree $\tau \in \T^{(\mathsf{e}, \mathsf{o})}$ we have
\begin{equation*}
\Omega^{(w^\mathsf{e}, w^\mathsf{o})}(\tau) = \prod_{u \in \mathsf{e}(\tau)} w^\mathsf{e}(k_{u}) \prod_{u \in \mathsf{o}(\tau)} w^\mathsf{o}(k_{u}) = \prod_{u \in \mathsf{o}(\tau)} w(\text{deg}(u)),
\end{equation*}
and Bj{\"o}rnberg \& Stef{\'a}nsson \cite[Theorem 3.1]{BJ14} have obtained a limit theorem for the measure $\Q^{(w^\mathsf{e}, w^\mathsf{o})}_n$ on $\T^{(\mathsf{e}, \mathsf{o})}_n$ as $n \to \infty$, in the local topology. Loosely speaking, the dual tree $T^\circ_n$ converges locally to a limiting infinite two-type tree which can be explicitly constructed, and which is in a certain sense a two-type Galton--Watson tree conditioned to survive.  We do not enter details as we will not use this and refer to \cite{BJ14} for precise statements and proofs.

\section{Non-crossing partitions as compact subsets of the unit disk}
\label{sec:Hausdorff}

We investigate in this section the asymptotic behavior, as $n \to \infty$, of a non-crossing partition sampled according to $\P^{\mu}_{n}$ and viewed as an element of the space of all compact subsets of the unit disk equipped with the Hausdorff distance. 

\paragraph{Main assumptions.} We restrict ourselves to the case where $\mu = (\mu(k), k \ge 0)$ defines a critical probability measure, i.e. $\sum_{k=0}^\infty \mu(k) = \sum_{k=0}^\infty k \mu(k) = 1$. Recall from Sec.~\ref{sec:sgn} that any sequence of weights $(w(k), k \ge 0)$ such that
\begin{equation*}
\rho = \left(\limsup_{k \rightarrow \infty} w(k)^{1/k}\right)^{-1} > 0 \qquad\text{and}\qquad \lim_{t \uparrow \rho} \frac{\sum_{k=0}^\infty k w(k) t^k}{\sum_{k=0}^\infty w(k) t^k} \ge 1
\end{equation*}
is equivalent to such a measure $\mu$ and then $\P^\mu_n=\P^w_n$ for every $n \ge 1$. We shall in addition assume that $\mu$ belongs to the domain of attraction of a stable law of index $\alpha \in (1, 2]$, i.e. either it has finite variance: $\sum_{k=0}^\infty k^2 \mu(k) < \infty$ (in the case $\alpha = 2$), or $\sum_{k=j}^\infty \mu(k) = j^{-\alpha} L(j)$, where $L$ is a slowly varying function at infinity.  Without further notice, we always assume that $\mu(0)+\mu(1)<1$ to discard degenerate cases.

 In this section, we shall establish the following result.
\begin{thm}\label{thm:convergence_PNC_lamination}
Fix  $\alpha \in (1,2]$. There exists a random compact subset of the unit disk $\mathbf{L}_{\alpha}$  such that for every critical offspring distribution  $\mu$  belonging to the domain of attraction of a stable law of index $\alpha$, if $P_{n}$ is a random non-crossing partition sampled according to $\P^\mu_n$, for every integer $n \geq 1$ such that $\P^\mu_n$ is well defined, the convergence
$$ P_n	\quad \mathop{\longrightarrow}^{(d)}_{n \rightarrow \infty} \quad \mathbf{L}_{\alpha}$$
holds in distribution for the Hausdorff distance on the space of all compact subsets of $\overline{\D}$.
\end{thm}

The random compact set  $\mathbf{L}_{\alpha}$ is a geodesic lamination; for $\alpha=2$, the  set $\mathbf{L}_{2}$ is Aldous' Brownian triangulation of the disk \cite{Ald94b}, while  $\mathbf{L}_{\alpha}$ is the $\alpha$-stable lamination introduced in \cite{Kor11} for  $\alpha \in (1,2)$.  Observe that Theorem \ref{thm:convergence_PNC_lamination} applies for uniform $\mathcal{A}$-constrained non-crossing partitions of $[n]$ when $ \mathcal{A} \neq  \{1\} $, since this law is $\P_n^{w_\mathcal{A}}$ where $w_{\mathcal{A}}(k) = 1$ if $k \in  \mathcal{A}$ and $w_{\mathcal{A}}(k) = 0$ otherwise; the equivalent probability distribution defined in Example \ref{ex:mesures_proba_equivalentes} is then critical and with finite variance and thus corresponds to $\alpha=2$.

Before explaining the construction of   $\mathbf{L}_{\alpha}$, we mention an interesting corollary. Recall from the Introduction the notation $\mathsf{C}(P_{n})$  for the (angular) length of the longest chord.

\begin{cor}\label{cor:asymp}Fix  $\alpha \in (1,2]$. There exists a random variable $ \mathsf{C}_{\alpha}$ such that for every critical offspring distribution  $\mu$  belonging to the domain of attraction of a stable law of index $\alpha$, if $P_{n}$ is a random non-crossing partition sampled according to $\P^\mu_n$, for every integer $n \geq 1$ such that $\P^\mu_n$ is well defined, the convergence
$$\mathsf{C}(P_{n})  \quad \mathop{\longrightarrow}^{(d)}_{n \rightarrow \infty} \quad \mathsf{C}_\alpha$$
holds in distribution.
\end{cor} 
This immediately follows from Theorem \ref{thm:convergence_PNC_lamination}, since the functional ``longest chord'' is continuous on the set of laminations. Aldous \cite{Ald94b} (see also \cite{DFHN99}) showed that the law of $\mathsf{C}_{2}$ has the following explicit distribution:
$$ \frac{1}{\pi} \frac{3x-1}{ x^{2}(1-x)^{2} \sqrt{1-2x}} \mathbbm{1}_{ \frac{1}{3} \leq x \leq \frac{1}{2}} dx.$$
See \cite{Shi14} for a study of the longest chord of stable laminations. As before, observe that Theorem \ref{thm:A}  follows from Corollary \ref{cor:asymp}, which  applies with $\alpha=2$ for uniform $\mathcal{A}$-constrained non-crossing partitions of $[n]$ when $ \mathcal{A}  \neq  \{1\}$.

\paragraph{Techniques.} We briefly comment on the main techniques involved in the proof of Theorem \ref{thm:convergence_PNC_lamination}. Since it is simple to recover $P_{n}$ from its dual two-type tree $T^{\circ}(P_{n})$, it seems natural to study scaling limits of $T^{\circ}(P_{n})$. However, this is not the road we  take: we rather code $P_{n}$ by the associated one-type tree $\mathcal{T}^\circ(P_{n})$, which, as we have earlier seen, has the law of a Galton--Watson tree with offspring distribution $\mu$ conditioned to have $n+1$ vertices, and is therefore simpler to study. We then follow the route of \cite{Kor11}: we code $\mathcal{T}^\circ(P_{n})$ via a discrete walk; the latter converges in distribution to a continuous-time process, we then define $\mathbf{L}_\alpha$ from this limit path and we show that it is indeed the limit of the discrete non-crossing partitions.

In \cite{Kor11}, it is shown that certain random dissections of $[n]$ (a dissection of a polygon with $n$ vertices is a collection of non-crossing diagonals) are shown to converge to the stable lamination, by using the fact that their dual trees are Galton--Watson trees conditioned to have a fixed number of leaves. Our arguments are similar to that of \cite[Sec.~2 and 3]{Kor11}, but the devil is in the details since the objects under consideration and their coding by trees are different: first, vertices with outdegree $1$ are forbidden in \cite{Kor11}, and second a vertex with outdegree $k$ in \cite{Kor11} corresponds to $k+1$ chords in the associated discrete lamination, whereas in our case a vertex with outdegree $k$ corresponds to $k$ chords in the associated non-crossing partition. In particular, the proofs of \cite[Sec.~2 and 3]{Kor11} do not carry out with mild modifications, and for this reason we give a complete proof of Theorem \ref{thm:convergence_PNC_lamination}.

\medskip

From now on, we fix $\alpha \in (1,2]$,  a critical offspring distribution $\mu$ belonging to the domain of attraction of a stable law of index $\alpha$, and we let $P_{n}$ be a random non-crossing partition sampled according to $\P^\mu_n$, for every integer $n \geq 1$ such that $\P^\mu_n$ is well defined.

\subsection{Non-crossing partitions and paths}
\label{section:cas_discret}

We first explain how a plane tree can be coded by a function, called {\L}ukasiewicz path, and then we describe how to define a non-crossing partition $P$ from the {\L}ukasiewicz path coding the tree $\mathcal{T}^\circ(P)$.

\begin{figure}[h!] \centering
\begin{scriptsize}
\begin{tikzpicture}
\draw[ultra thin, dashed]	(0,0) circle (2);
\foreach \x in {1, 2, ..., 12}
	\coordinate (\x) at (-\x*360/12 : 2);
\foreach \x in {1, 2, ..., 12}
	\draw
	[fill=black]	(\x) circle (1pt)
	(-\x*360/12 : 2*1.1) node {\x}
;
\filldraw[pattern=north east lines]
	(1) -- (3) -- (5) -- cycle
	(6) -- (7) -- (11) -- (12) -- cycle
	(9) -- (10)
;
\end{tikzpicture}
\qquad\qquad
%
\begin{tikzpicture}
\coordinate (0) at (0,0);
	\coordinate (1) at (-1.5,1);
		\coordinate (11) at (-1.5,2);
	\coordinate (2) at (0,1);
		\coordinate (21) at (0,2);
	\coordinate (3) at (1.5,1);
		\coordinate (31) at (.5,2);
		\coordinate (32) at (1.25,2);
			\coordinate (321) at (1.25,3);
				\coordinate (3211) at (.75,4);
				\coordinate (3212) at (1.75,4);
		\coordinate (33) at (1.75,2);
		\coordinate (34) at (2.5,2);
\draw
	(0) -- (1) -- (11)
	(0) -- (2) -- (21)
	(0) -- (3)
	(3) -- (31)	(3) -- (32)	(3) -- (33)	(3) -- (34)
	(32) -- (321) -- (3211)	(321) -- (3212)
;
\draw[fill=black]
	(0) circle (1pt)
	(1) circle (1pt)
	(2) circle (1pt)
	(3) circle (1pt)
	(32) circle (1pt)
	(321) circle (1pt)
	(11) circle (1pt)
	(21) circle (1pt)
	(31) circle (1pt)
	(33) circle (1pt)
	(34) circle (1pt)
	(3211) circle (1pt)
	(3212) circle (1pt)
;
%
\draw
	(0) node[below] {$0$}
	(1) node[left] {$1$}
	(11) node[left] {$2$}
	(2) node[left] {$3$}
	(21) node[left] {$4$}
	(3) node[right] {$5$}
	(31) node[right] {$6$}
	(32) node[right] {$7$}
	(321) node[left] {$8$}
	(3211) node[left] {$9$}
	(3212) node[left] {$10$}
	(33) node[right] {$11$}
	(34) node[right] {$12$}
;
\end{tikzpicture}
\end{scriptsize}
\caption{The partition $P = \{\{1, 3, 5\}, \{2\}, \{4\}, \{6, 7, 11, 12\}, \{8\}, \{9, 10\}\}$ and the tree $\mathcal{T}^\circ(P)$.}
\end{figure}
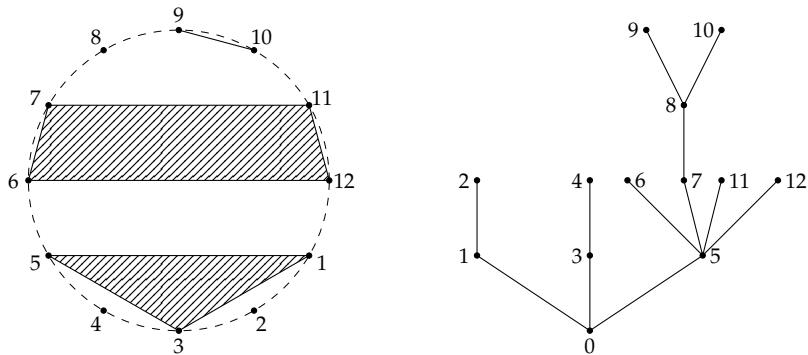

Let $\tau \in \T_{n+1}$ and $\varnothing = u(0) \prec u(1) \prec \dots \prec u(n)$ its vertices, listed in lexicographical order. Recall that $k_u$ denotes the number of children of $u \in \tau$. The {\L}ukasiewicz path $\W(\tau) = (\W_j(\tau), 0 \le j \le n+1)$ of $\tau$ is defined by $\W_0(\tau) = 0$ and
\begin{equation*}
\W_{j+1}(\tau) = \W_j(\tau) + k_{u(j)}(\tau)-1 \qquad\text{for every } 0 \le j \le n.
\end{equation*}
One easily checks that $\W_j(\tau) \ge 0$ for every $0 \le j \le n$ but $\W_{n+1}(\tau)=-1$. Observe that for every $0 \le j \le n$, $\W_{j+1}(\tau) - \W_j(\tau) \ge -1$, with equality if and only if $u(j)$ is a leaf of $\tau$. The next result, whose proof is left as an exercise, explains how to reconstruct a plane tree from its {\L}ukasiewicz path.

\begin{prop}\label{prop:bijection_arbre_marche}
Let $\tau \in \T_{n+1}$, $\varnothing = u(0) \prec u(1) \prec \dots \prec u(n)$ its vertices listed in lexicographical order and $\W(\tau)$ its {\L}ukasiewicz path. Fix $0 \le j \le n-1$ such that $k \coloneqq k_{u(j)}(\tau) \ge 1$. Let $s_1, \dots, s_k \in \{1, \dots, n\}$ be defined by $s_\ell = \inf\{m \ge j+1 : \W_m(\tau) = \W_{j+1}(\tau) - (\ell-1)\}$ for $1 \le \ell \le k$ (in particular, $s_1 = j+1$). Then $u(s_1), u(s_2), \dots, u(s_k)$ are the children of $u(j)$ listed in lexicographical order.
\end{prop}

\begin{figure}[h!] \centering
\begin{scriptsize}
\begin{tikzpicture}
\coordinate (0) at (0,0);
	\coordinate (1) at (-1.5,1);
		\coordinate (11) at (-1.5,2);
	\coordinate (2) at (0,1);
		\coordinate (21) at (0,2);
	\coordinate (3) at (1.5,1);
		\coordinate (31) at (.5,2);
		\coordinate (32) at (1.25,2);
			\coordinate (321) at (1.25,3);
				\coordinate (3211) at (.75,4);
				\coordinate (3212) at (1.75,4);
		\coordinate (33) at (1.75,2);
		\coordinate (34) at (2.5,2);
\draw
	(0) -- (1) -- (11)
	(0) -- (2) -- (21)
	(0) -- (3)
	(3) -- (31)	(3) -- (32)	(3) -- (33)	(3) -- (34)
	(32) -- (321) -- (3211)	(321) -- (3212)
;
\draw[fill=black]
	(0) circle (1pt)
	(1) circle (1pt)
	(2) circle (1pt)
	(3) circle (1pt)
	(32) circle (1pt)
	(321) circle (1pt)
	(11) circle (1pt)
	(21) circle (1pt)
	(31) circle (1pt)
	(33) circle (1pt)
	(34) circle (1pt)
	(3211) circle (1pt)
	(3212) circle (1pt)
;
%
\draw
	(0) node[below] {$0$}
	(1) node[left] {$1$}
	(11) node[left] {$2$}
	(2) node[left] {$3$}
	(21) node[left] {$4$}
	(3) node[right] {$5$}
	(31) node[right] {$6$}
	(32) node[right] {$7$}
	(321) node[left] {$8$}
	(3211) node[left] {$9$}
	(3212) node[left] {$10$}
	(33) node[right] {$11$}
	(34) node[right] {$12$}
;
\end{tikzpicture}
\quad
%
\begin{tikzpicture}
\draw[thin, ->]	(0,0) -- (.75*14,0);
\draw[thin, ->]	(0,-1) -- (0,3.5);
\foreach \x in {1, 2, ..., 13}
	\draw (.75*\x,.05)--(.75*\x,-.05)	(.75*\x,0) node[below] {$\x$};
\foreach \x in {-1, 0, 1, ..., 3}
	\draw (.05,\x)--(-.05,\x)	(0,\x) node[left] {$\x$};
\draw[fill=black]
	(0, 0) circle (1.25pt)
	(.75*1, 2) circle (1.25pt)
	(.75*2, 2) circle (1.25pt)
	(.75*3, 1) circle (1.25pt)
	(.75*4, 1) circle (1.25pt)
	(.75*5, 0) circle (1.25pt)
	(.75*6, 3) circle (1.25pt)
	(.75*7, 2) circle (1.25pt)
	(.75*8, 2) circle (1.25pt)
	(.75*9, 3) circle (1.25pt)
	(.75*10, 2) circle (1.25pt)
	(.75*11, 1) circle (1.25pt)
	(.75*12, 0) circle (1.25pt)
	(.75*13, -1) circle (1.25pt)
;
\end{tikzpicture}
\end{scriptsize}
\caption{A plane tree and its {\L}ukasiewicz path.}
\label{fig:marche_Luka}
\end{figure}
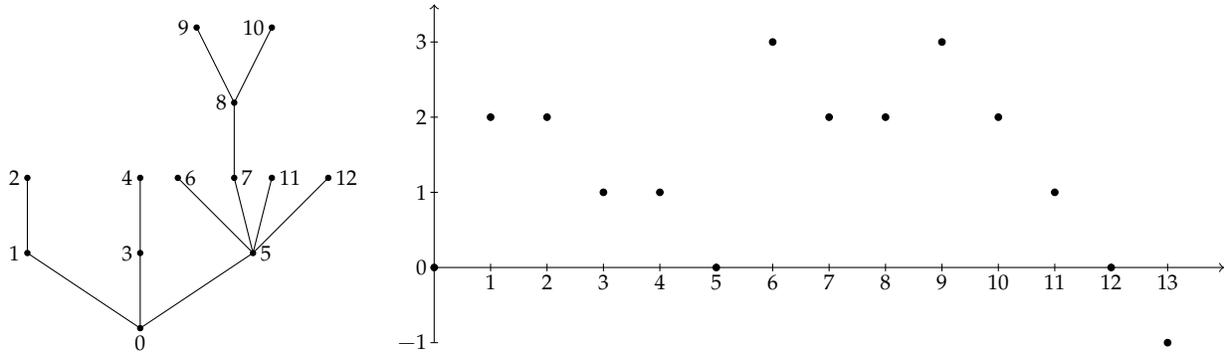

We now describe how to define a non-crossing partition from a {\L}ukasiewicz path. Fix $n \in \N$ and $W = (W_j, 0 \le j \le n+1)$ a path such that $W_0 = 0$, for every $0 \le j \le n$, $W_{j+1} - W_j \ge -1$ with the condition that $W_j \ge 0$ for every $0 \le j \le n$ and $W_{n+1} = -1$. 
Define
\begin{equation*}
k_j = W_{j+1} - W_j + 1 \qquad\text{for every } 0 \le j \le n-1;
\end{equation*}
if $k_j \ge 1$, then let
\begin{equation*}
s^j_\ell = \inf\{m \ge j + 1 : W_m = W_{j+1}-(\ell-1)\} \qquad\text{for every } 1 \le \ell \le k_j,
\end{equation*}
and then set $s^j_{k_j+1} = s^j_1 = j+1$. Next define  $\mathbf{P}(W)$ by
\begin{equation}\label{eq:def_pnc_via_marche}
\mathbf{{P}}(W) = \bigcup_{j : k_j \ge 1} \bigcup_{\ell=1}^{k_j} \left[\exp\left(-2\mathrm{i}\pi \frac{s^j_\ell}{n}\right), \exp\left(-2\mathrm{i}\pi \frac{s^j_{\ell+1}}{n}\right)\right].
\end{equation}

Let us briefly explain what this means: if $W$ is the {\L}ukasiewicz path of a tree $\tau$ with its vertices labelled as above, then $k_{j}$ is the number of children of $u(j)$, and $s_{1}^{j}, \ldots, s_{k_{j}}^{j}$ are the indices of its children. Recall from Sec.~\ref{sec:pp} that from a tree $\tau$, we can define a non-crossing partition $P_\circ(\tau)$ by joining two consecutive children in $\tau$ (where the first and the last ones are consecutive by convention); this is exactly what is done in \eqref{eq:def_pnc_via_marche}. Recall also from Sec.~\ref{sec:pp} the construction of the plane tree tree $\mathcal{T}^\circ(P)$ from a non-crossing partition $P$.

\begin{prop}\label{prop:bijection_marche_pnc_et_arbres_nc}
For every non-crossing partition $P$,  we have
\begin{equation*}
P = \mathbf{P}(\mathcal{W}(\mathcal{T}^\circ(P))).
\end{equation*}
\end{prop}

\begin{proof}
To simplify notation,  let $\mathcal{W}^\circ$ denote the {\L}ukasiewicz path of $\mathcal{T}^\circ(P)$ and let $n$ be its length. First, note that $\mathbf{P}(\mathcal{W}^\circ)$ is a partition of $[n]$: with the notation used in \eqref{eq:def_pnc_via_marche}, the blocks are given by the sets $\{s^j_1, \dots, s^j_{k_j}\}$ for the $j$'s such that $k_j \ge 1$. To show that it is non-crossing, fix $j, j' \in \{0, \dots, n-1\}$ with $k_j, k_{j'} \ge 1$ and fix $\ell \in \{1, \dots, k_j+1\}$ and $\ell' \in \{1, \dots, k_{j'}+1\}$ with $(j, \ell) \ne (j', \ell')$; one checks that the intervals $(s^j_\ell, s^j_{\ell+1})$ and $(s^{j'}_{\ell'}, s^{j'}_{\ell'+1})$ either are disjoint or one is included in the other so that the chords
\begin{equation*}
\left[\exp\left(-2\mathrm{i}\pi \frac{s^j_\ell}{n}\right), \exp\left(-2\mathrm{i}\pi \frac{s^j_{\ell+1}}{n}\right)\right]
\quad\text{and}\quad
\left[\exp\left(-2\mathrm{i}\pi \frac{s^{j'}_{\ell'}}{n}\right), \exp\left(-2\mathrm{i}\pi \frac{s^{j'}_{{\ell'}+1}}{n}\right)\right]
\end{equation*}
do not cross. Further, as explained above, by construction, the chords of $ \mathbf{P}(\mathcal{W}^\circ)$ are chords between consecutive children of $\mathcal{T}^\circ(P)$. The equality $P = \mathbf{P}(\mathcal{W}^\circ)$ then  simply follows from the fact that, by construction and Prop.~\ref{prop:bijection_pnc_arbres}, $i,j \in [n]$ belong to the same block of $P$ if and only if $u(i)$ and $u(j)$ have the same parent in $\mathcal{T}^\circ(P)$.
\end{proof}

As previously explained, we will prove the convergence, when $n \to \infty$, of a random non-crossing partition $P_n$ of $[n]$ sampled according to $\P^\mu_n$, by looking at the scaling limit of the {\L}ukasiewicz path of the conditioned Galton--Watson tree $\mathcal{T}^\circ(P_{n})$. The latter is known (see Thm.~\ref{thm:convergence_marche_excursion} below) to be the normalized excursion of a spectrally positive strictly $\alpha$-stable Lévy process $X^{\rm ex}_\alpha$ which we next introduce. The main advantage of this approach is that  $\mathcal{T}^\circ(P_{n})$ is a (conditioned) one-type Galton--Watson tree, whereas the dual tree $T^{\circ}(P_{n})$ of $P_{n}$ is a (conditioned) two-type Galton--Watson tree. We mention here that \cite{Abr13} uses a ``modified'' {\L}ukasiewicz path to study a two-type Galton--Watson tree; actually this path is just the  {\L}ukasiewicz path of the one-type tree associated with the two-type tree by the Janson--Stef\'ansson bijection.

\subsection{Convergence to the stable excursion}
\label{sec:convergence_excursion}

Fix $\alpha \in (1, 2]$ and consider a strictly stable spectrally positive L{\'e}vy process of index $\alpha$: $X_\alpha$ is a random process with paths in the set $\D([0, \infty), \R)$ of c{\`a}dl{\`a}g functions endowed with the Skorokhod $J_{1}$ topology (see e.g. Billingsley \cite{Bil99} for details on this space) which has independent and stationary increments, no negative jumps and such that
$\Es{\exp(-\lambda X_\alpha(t))} = \exp(t \lambda^\alpha) $ for every $t, \lambda > 0$. Using excursion theory, it is then possible to define $X^{\rm ex}_\alpha$, the normalized excursion of $X_\alpha$, which is a random variable with values in $\D([0, 1], \R)$, such that $X^{\rm ex}_\alpha(0) = X^{\rm ex}_\alpha(1) = 0$ and, almost surely, $X^{\rm ex}_\alpha(t) > 0$ for every $t \in (0,1)$. We do not enter into details, and refer the interested reader to Bertoin \cite{Ber96} for details on Lévy processes and Chaumont \cite{CH97} for interesting ways to obtain such a process by path transformations.

An important point is that $X^{\rm ex}_\alpha$ is continuous for $\alpha=2$, and indeed $X^{\rm ex}_2 / \sqrt{2}$ is the standard Brownian excursion, whereas the set of discontinuities of $X^{\rm ex}_\alpha$ is dense in $[0,1]$ for every $\alpha \in (1,2)$; we shall treat the two cases separately. Duquesne \cite[Prop.~4.3 and proof of Theorem 3.1]{Du03} provides the following limit theorem which is the steppingstone of our results in this section.

\begin{thm}[Duquesne \cite{Du03}]\label{thm:convergence_marche_excursion}
Fix $\alpha \in (1, 2]$ and let $(\mu(k), k \ge 0)$ be a critical probability measure in the domain of attraction of a stable law of index $\alpha$. For every integer $n$ such that $\Q^\mu_{n+1}$ is well defined, sample $\tau_n$ according to $\Q^\mu_{n+1}$. Then there exists a sequence $(B_n)_{n \ge 1}$ of positive constants converging to $\infty$ such that the convergence
\begin{equation*}
 \left(  \frac{\W_{\lfloor n s \rfloor}(\tau_n)}{B_{n}}; 0 \leq s \leq 1\right)   \quad \mathop{\longrightarrow}^{(d)}_{n \rightarrow \infty} \quad  \left(X^{\rm ex}_\alpha(s); 0 \leq s \leq 1\right)
\end{equation*}
holds in distribution for the Skorokhod topology on $\D([0, 1], \R)$.
\end{thm}

Recall that if we sample $P_n$ according to $\P^\mu_{n}$, then the plane tree $\mathcal{T}^\circ(P_n)$ is distributed according to $\Q^\mu_{n+1}$. Thus,  denoting by $\W^n=\W(\mathcal{T}^\circ(P_n))$ the {\L}ukasiewicz path of $\mathcal{T}^\circ(P_n)$, the convergence
\begin{equation}\label{eq:convergence_marche_excursion}
 \left(  \frac{\W^{n}_{\lfloor n s \rfloor}}{B_{n}}; 0 \leq s \leq 1\right)   \quad \mathop{\longrightarrow}^{(d)}_{n \rightarrow \infty} \quad  \left(X^{\rm ex}_\alpha(s); 0 \leq s \leq 1\right)
\end{equation}
holds
in distribution for the Skorokhod topology on $\D([0, 1], \R)$.

We next define continuous laminations by replacing the {\L}ukasiewicz path by $X^{\rm ex}_\alpha$  and mimicking the definition \eqref{eq:def_pnc_via_marche}. We prove, using \eqref{eq:convergence_marche_excursion}, that they are the limit of $P_n$ as $n \to \infty$. We first consider the case $\alpha=2$ as a warm-up before treating the more involved the case $\alpha \in (1, 2)$.

\subsection{The Brownian case}
\label{sec:cas_brownien}

Let $\mathbbm{e} = X^{\rm ex}_2$; we define an equivalence relation $\overset{\mathbbm{e}}{\thicksim}$ on $[0,1]$ as follows: for every $s, t \in [0, 1]$, we set $s \overset{\mathbbm{e}}{\thicksim} t$ when $\mathbbm{e}(s \wedge t) = \mathbbm{e}(s \vee t) = \min_{[s \wedge t, s \vee t]} \mathbbm{e}$. We then define a subset of $\overline{\D}$ by
\begin{equation}\label{eq:triangulation_brownienne}
\mathbf{L}(\mathbbm{e}) \coloneqq \bigcup_{s \overset{\mathbbm{e}}{\thicksim} t} \left[\mathrm{e}^{-2\mathrm{i}\pi s}, \mathrm{e}^{-2\mathrm{i}\pi t}\right].
\end{equation}
Using the fact that, almost surely, $\mathbbm{e}$ is continuous and its local minima are distinct, one can prove (see Aldous \cite{Ald94b} and  Le Gall \& Paulin \cite{LGP08}) that almost surely, $\mathbf{L}(\mathbbm{e})$ is a geodesic lamination of $\overline{\D}$ and that, furthermore, it is maximal for the inclusion relation among geodesic laminations of $\overline{\D}$. Observe that $s \overset{\mathbbm{e}}{\thicksim} s$ for every $s \in [0, 1]$ so $\S \subset \mathbf{L}(\mathbbm{e})$. Also, since $\mathbf{L}(\mathbbm{e})$ is maximal, its faces, i.e. the connected components of $\overline{\D} \setminus \mathbf{L}(\mathbbm{e})$, are open triangles whose vertices belong to $\S$; $\mathbf{L}(\mathbbm{e})$ is called the Brownian triangulation and corresponds to $\mathbf{L}_2$ in Theorem \ref{thm:convergence_PNC_lamination}.

\begin{proof}[Proof of Theorem \ref{thm:convergence_PNC_lamination} for $\alpha=2$]
Using Skorokhod’s representation theorem, we assume that the convergence \eqref{eq:convergence_marche_excursion} holds almost surely with $\alpha = 2$; we then fix $\omega$ in the probability space such that this convergence holds for $\omega$. Since the space of compact subsets of $\overline{\D}$ equipped with the Hausdorff distance is compact, we have the convergence, along a subsequence (which depends on $\omega$), of $P_n$ to a limit $L_\infty$, and it only remains to show that $L_\infty = \mathbf{L}(\mathbbm{e})$. Observe first that, since the space of geodesic laminations of $\overline{\D}$ is closed, $L_\infty$ is a lamination. Then, by maximality of $\mathbf{L}(\mathbbm{e})$, it suffices to prove that $\mathbf{L}(\mathbbm{e}) \subset L_\infty$ to obtain the equality of these two sets.

Fix $\varepsilon > 0$ and $0 \le s < t \le 1$ such that $s \overset{\mathbbm{e}}{\thicksim} t$. Using the convergence \eqref{eq:convergence_marche_excursion} and the properties of the Brownian excursion ({namely that times of local minima are almost surely dense in $[0,1]$}), we can find integers $j_n, l_n \in \{1, \dots, n - 1\}$ such that every $n$ large enough, we have
\begin{equation*}
|n^{-1} j_n - s| < \varepsilon, \quad
|n^{-1} l_n - t| < \varepsilon, \quad
\W^n_{j_n} > \W^n_{j_n-1}, \quad\text{and}\quad
l_n = \min\{m > j_n : \W^n_m < \W^n_{j_n}\}.
\end{equation*}
In other words, $u(j_{n})$ and $u(l_{n})$ are consecutive children of $u(j_{n}-1)$ in $\mathcal{T}^\circ(P_n)$. By Prop.~\ref{prop:bijection_marche_pnc_et_arbres_nc}, the last two properties yield
\begin{equation*}
\left[\exp\left(-2\mathrm{i}\pi \frac{j_n}{n}\right), \exp\left(-2\mathrm{i}\pi \frac{l_n}{n}\right)\right] \subset P_n.
\end{equation*}
Thus, for every $n$ large enough, the chord $[\mathrm{e}^{-2\mathrm{i}\pi s}, \mathrm{e}^{-2\mathrm{i}\pi t}]$ lies within distance $2\varepsilon$ from $P_n$. Letting $n \to \infty$, along a subsequence, we obtain that $[\mathrm{e}^{-2\mathrm{i}\pi s}, \mathrm{e}^{-2\mathrm{i}\pi t}]$ lies within distance $2\varepsilon$ from $L_\infty$. As $\varepsilon$ is arbitrary, we have $[\mathrm{e}^{-2\mathrm{i}\pi s}, \mathrm{e}^{-2\mathrm{i}\pi t}] \subset L_\infty$, hence $\mathbf{L}(\mathbbm{e}) \subset L_\infty$ and the proof is complete.
\end{proof}

\subsection{The stable case}

We follow the presentation of \cite{Kor11}. Fix $\alpha \in (1, 2)$ and consider $X^{\rm ex}_\alpha$ the normalized excursion of the $\alpha$-stable Lévy process. For every $t \in (0, 1]$, we denote by $\Delta X^{\rm ex}_\alpha(t) = X^{\rm ex}_\alpha(t) - X^{\rm ex}_\alpha(t-) \ge 0$ its jump at $t$, and we set $\Delta X^{\rm ex}_\alpha(0) = X^{\rm ex}_\alpha(0-) = 0$. 
We recall from \cite[Prop.~2.10]{Kor11} that $X^{\rm ex}_\alpha$ fulfills the following four properties with probability one:
\begin{itemize}
\item[(H1)] For every $0 \le s < t \le 1$, there exists at most one value $r \in (s, t)$ such that $X^{\rm ex}_\alpha(r) = \inf_{[s, t]} X^{\rm ex}_\alpha$.
\item[(H2)] For every $t \in (0,1)$ such that $\Delta X^{\rm ex}_\alpha(t) > 0$, we have $\inf_{[t, t+\varepsilon]} X^{\rm ex}_\alpha < X^{\rm ex}_\alpha(t)$ for every $0 < \varepsilon \le 1-t$;
\item[(H3)] For every $t \in (0,1)$ such that $\Delta X^{\rm ex}_\alpha(t) > 0$, we have $\inf_{[t-\varepsilon, t]} X^{\rm ex}_\alpha < X^{\rm ex}_\alpha(t-)$ for every $0 < \varepsilon \le t$;
\item[(H4)] For every $t \in (0,1)$ such that $X^{\rm ex}_\alpha$ attains a local minimum at $t$ (which implies $\Delta X^{\rm ex}_\alpha(t) = 0$), if $s = \sup\{ u \in [0,t] : X^{\rm ex}_\alpha(u) < X^{\rm ex}_\alpha(t)\}$, then $\Delta X^{\rm ex}_\alpha(s) > 0$ and $X^{\rm ex}_\alpha(s-) < X^{\rm ex}_\alpha(t) < X^{\rm ex}_\alpha(s)$.
\end{itemize}
We will always implicitly discard the null-set for which at least one of these properties does not hold. We next define a relation (not equivalence relation in general) on $[0, 1]$ as follows: for every $0 \le s < t \le 1$, we set
\begin{equation*}
s \simeq^{X^{\rm ex}_\alpha} t \quad\text{if}\quad t = \inf\{u > s : X^{\rm ex}_\alpha(u) \le X^{\rm ex}_\alpha(s-)\},
\end{equation*}
and then for $0 \le t < s \le 1$, we set $s \simeq^{X^{\rm ex}_\alpha} t$ if $t \simeq^{X^{\rm ex}_\alpha} s$, and finally we agree that $s \simeq^{X^{\rm ex}_\alpha} s$ for every $s \in [0,1]$. We next define the following subset of $\overline{\D}$:
\begin{equation}\label{eq:lamination_et_triangulation_stable}
\mathbf{L}_{\alpha} \coloneqq \bigcup_{s \simeq^{X^{\rm ex}_\alpha} t} \left[\mathrm{e}^{-2\mathrm{i}\pi s}, \mathrm{e}^{-2\mathrm{i}\pi t}\right].
\end{equation}
Observe that $\S \subset \mathbf{L}_{\alpha}$. Using the above properties, it is proved in \cite[Prop.~2.9]{Kor11} that $\mathbf{L}_{\alpha}$ is a geodesic lamination of $\overline{\D}$, called the $\alpha$-stable lamination. The latter is not maximal: each face is bounded by infinitely many chords (the intersection of the closure of each face and the unit disk has indeed a non-trivial Hausdorff dimension in the plane).

We next prove Theorem \ref{thm:convergence_PNC_lamination}; as in the case $\alpha=2$, we assume using Skorokhod’s representation theorem that \eqref{eq:convergence_marche_excursion} holds almost surely and we work with $\omega$ fixed in the probability space such that this convergence (as well as the properties (H1) to (H4)) holds for $\omega$. To simplify notation, we set
\begin{equation*}
X^n(s) = \frac{1}{B_{n}} \W^n_{\lfloor n s \rfloor} \qquad\text{for every } s \in [0,1].
\end{equation*}
Along a subsequence (which depends on $\omega$), we have the convergence of $P_n$ to a limit $L_\infty$, which is a lamination. It only remains to prove the identity $L_\infty = \mathbf{L}_\alpha$. {To do so, we shall prove the inclusions $\mathbf{L}_\alpha \subset L_\infty$ and $L_\infty \subset \mathbf{L}_\alpha$  in two separate lemmas.}

\begin{lem}\label{lem:convergence_pnc_et_arbres_cas_stable_1}
We have $\mathbf{L}_\alpha \subset L_\infty$.
\end{lem}

\begin{proof}
Notice that if $s < t$ and $s \simeq^{X^{\rm ex}_\alpha} t$, then $X^{\rm ex}_\alpha(t) = X^{\rm ex}_\alpha(s-)$ and $X^{\rm ex}_\alpha(r) > X^{\rm ex}_\alpha(s-)$ for every $r \in (s, t)$, hence $s \simeq^{X^{\rm ex}_\alpha} t$ if and only if one the following cases holds:
\begin{enumerate}
\item $\Delta X^{\rm ex}_\alpha(s) > 0$ and $t = \inf\{u > s : X^{\rm ex}_\alpha(u) = X^{\rm ex}_\alpha(s-)\}$, we write $(s,t) \in \mathcal{E}_1(X^{\rm ex}_\alpha)$;
\item $\Delta X^{\rm ex}_\alpha(s) = 0$, $X^{\rm ex}_\alpha(s) = X^{\rm ex}_\alpha(t)$ and $X^{\rm ex}_\alpha(r) > X^{\rm ex}_\alpha(s)$ for every $r \in (s, t)$, we write $(s,t) \in \mathcal{E}_2(X^{\rm ex}_\alpha)$.
\end{enumerate}
Using the observation (\cite[Prop.~2.14]{Kor11}) that, almost surely, for every pair $(s, t) \in \mathcal{E}_2(X^{\rm ex}_\alpha)$ and every $\varepsilon \in (0, (t-s)/2)$, there exists $s' \in [s,s+\varepsilon]$ and $t' \in [t-\varepsilon, t]$ with $(s', t') \in \mathcal{E}_1(X^{\rm ex}_\alpha)$, one can prove (\cite[Prop.~2.15]{Kor11}) that almost surely
\begin{equation}\label{eq:densite_temps_de_sauts_lamination_stable}
\mathbf{L}_\alpha = \overline{\bigcup_{(s, t) \in \mathcal{E}_1(X^{\rm ex}_\alpha)} \left[\mathrm{e}^{-2\mathrm{i}\pi s}, \mathrm{e}^{-2\mathrm{i}\pi t}\right]}.
\end{equation}
The proof thus reduces to showing that, for any $0 \le u < v \le 1$ such that $\Delta X^{\rm ex}_\alpha(u) > 0$ and $v = \inf\{w \ge u : X^{\rm ex}_\alpha(w) = X^{\rm ex}_\alpha(u-)\}$ fixed, we have $[\mathrm{e}^{-2\mathrm{i}\pi u}, \mathrm{e}^{-2\mathrm{i}\pi v}] \subset L_\infty$. Further, as in the case $\alpha=2$, it is sufficient to find sequences $u_n \to u$ and $v_n \to v$ as $n \to \infty$ such that for every $n$ large enough, $[\mathrm{e}^{-2\mathrm{i}\pi u_n}, \mathrm{e}^{-2\mathrm{i}\pi v_n}] \subset P_n$. Informally, the main difference with \cite{Kor11} is that we choose different sequences $u_{n},v_{n}$:  with the notation used in \eqref{eq:def_pnc_via_marche}, we shall take the pair $(u_n, v_n)$ of the form $n^{-1} (s^j_1, s^j_{k_j})$ for a certain $j$.

More precisely, fix $\varepsilon > 0$ and observe that, since $v$ cannot be a time of local minimum of $X^{\rm ex}_\alpha$ by (H4), then
\begin{equation*}
\inf_{[v-\varepsilon, v+\varepsilon]} X^{\rm ex}_\alpha < X^{\rm ex}_\alpha(v) = X^{\rm ex}_\alpha(u-) < \inf_{[u, v-\varepsilon]} X^{\rm ex}_\alpha.
\end{equation*}
Using the convergence \eqref{eq:convergence_marche_excursion}, we can then find a sequence $(u_n)_{n \ge 1}$ such that for every $n$ sufficiently large, we have
\begin{equation*}
u_n \in (u - \varepsilon, u + \varepsilon) \cap n^{-1} \N \quad\text{and}\quad
\inf_{[v-\varepsilon, v+\varepsilon]} X^n < X^n(u_n-) < \inf_{[u_n, v-\varepsilon]} X^n.
\end{equation*}
Define then $v_n \coloneqq \inf\{r \ge u_n : X^n(r) = X^n(u_n-)\}$ and observe that $v_n \in (v - \varepsilon, v + \varepsilon) \cap n^{-1}\N$. Moreover, {as $B_{n} X^{n}(u_{n})= \mathcal{W}^{n}_{nu_{n}} $ and $B_{n} X^{n}(u_{n}-)= \mathcal{W}^{n}_{nu_{n}-1} $}, we have $\W^n_{nu_n - 1} \le \W^n_{nu_n}$ and
\begin{equation*}
nv_n = \inf\{l \ge nu_n : \W^n_l = \W^n_{nu_n} - (\W^n_{nu_n} - \W^n_{nu_n - 1})\}.
\end{equation*}
We conclude from Prop.~\ref{prop:bijection_marche_pnc_et_arbres_nc} that
\begin{equation*}
\left[\mathrm{e}^{-2\mathrm{i}\pi u_n}, \mathrm{e}^{-2\mathrm{i}\pi v_n}\right] \subset P_n
\end{equation*}
for every $n$ large enough and the proof is complete.
\end{proof}

Finally, we end the proof of Theorem \ref{thm:convergence_PNC_lamination} with the converse inclusion.

\begin{lem}\label{lem:convergence_pnc_et_arbres_cas_stable_3}
We have $L_\infty \subset \mathbf{L}_\alpha$.
\end{lem}

\begin{proof}
Recall that $L_\infty$ is the limit of $P_n$ along a subsequence, say, $(n_k)_{k \ge 1}$. Let us rewrite \eqref{eq:def_pnc_via_marche}, combined with Prop.~\ref{prop:bijection_marche_pnc_et_arbres_nc}, as
\begin{equation*}
P_{n_k} = \bigcup_{(u, v) \in \mathcal{E}_{(n_k)}} \left[\mathrm{e}^{-2\mathrm{i}\pi u}, \mathrm{e}^{-2\mathrm{i}\pi v}\right],
\end{equation*}
where $\mathcal{E}_{(n_k)}$ is a symmetric finite subset of $[0,1]^2$. Upon extracting a further subsequence, we may, and do, assume that $\mathcal{E}_{(n_k)}$ converges in the Hausdorff sense as $k \to \infty$ to a symmetric closed subset $\mathcal{E}_{\infty}$ of $[0,1]^2$. One then checks that
\begin{equation*}
L_\infty = \bigcup_{(u, v) \in \mathcal{E}_{\infty}} \left[\mathrm{e}^{-2\mathrm{i}\pi u}, \mathrm{e}^{-2\mathrm{i}\pi v}\right].
\end{equation*}
It only remains to prove that every pair $(u, v) \in \mathcal{E}_\infty$ satisfies $u \simeq^{X^{\rm ex}_\alpha} v$. Fix $(u, v) \in \mathcal{E}_\infty$ with $u < v$; we aim to show that $v = \inf\{r > u : X^{\rm ex}_\alpha(r) \le X^{\rm ex}_\alpha(u-)\}$.

For every integer $j \in \{{1}, \dots, n\}$ and let $p(j)$ be the index of the parent of vertex labelled $j$ in $\mathcal{T}^\circ(P_n)$: $p(j) = \sup\{m < j : \W^n_m \le \W^n_j\}$. Observe then that $[\mathrm{e}^{-2\mathrm{i}\pi j_n/n}, \mathrm{e}^{-2\mathrm{i}\pi l_n/n}] \subset P_n$ when $p(j_n) = p(l_n)$ and, either $l_n = \inf\{m \ge j_n : \W^n_m = \W^n_{j_n}-1\}$, or $j_n = p(j_n)+1$ and $l_n = \inf\{m \ge j_n : \W^n_m = \W^n_{p(j_n)}\}$.

By definition, $(u, v)$ is the limit as $k \to \infty$ of elements $(u_{n_k}, v_{n_k})$ in $\mathcal{E}_{(n_k)}$. Upon extracting a subsequence, we may, and do, suppose that either each pair $(j_{n_k}, l_{n_k}) = (n_k u_{n_k}, n_k v_{n_k})$ fulfills the first condition above, or they all fulfill the second one. We first focus on the first case. We therefore suppose that we can find integers $j_{n_k} < l_{n_k}$ in $\{{1}, \dots, n_k\}$ such that
\begin{equation*}
(u, v) = \lim_{k \to \infty} \left(\frac{j_{n_k}}{n_k}, \frac{l_{n_k}}{n_k}\right) \qquad\text{and}\qquad
l_{n_k} = \inf\{m \ge j_{n_k} : \W^{n_k}_m = \W^{n_k}_{j_{n_k}}-1\} {\textrm{ for every } k \geq 1.}
\end{equation*}
We see that
\begin{equation}\label{eq:pas_d_idee_1}
X^{n_k}(r) \ge X^{n_k}\left(\frac{j_{n_k}}{n_k}\right) = X^{n_k}\left(\frac{l_{n_k}-1}{n_k}\right)
\qquad\text{for every } r \in \left[\frac{j_{n_k}}{n_k}, \frac{l_{n_k}-1}{n_k}\right],
\end{equation}
which yields, together with the {functional} convergence $X^n \to X^{\rm ex}_\alpha$,
\begin{equation}\label{eq:pas_d_idee_2}
X^{\rm ex}_\alpha(r) \ge X^{\rm ex}_\alpha(v-)
\qquad\text{for every } r \in (u, v).
\end{equation}
By (H3), we must have $\Delta X^{\rm ex}_\alpha(v) = 0$ and so $X^{n_k}(n_k^{-1} (l_{n_k}-1)) \to X^{\rm ex}_\alpha(v)$ as $k \to \infty$. On the other hand, the only possible accumulation points of $X^{n_k}(n_k^{-1} j_{n_k})$ are $X^{\rm ex}_\alpha(u-)$ and $X^{\rm ex}_\alpha(u)$.

We consider two cases. Suppose first that $\Delta X^{\rm ex}_\alpha(u) = 0$; then $X^{n_k}(n_k^{-1} j_{n_k}) \to X^{\rm ex}_\alpha(u)$ as $k \to \infty$ and it follows from \eqref{eq:pas_d_idee_1} that $X^{\rm ex}_\alpha(u) = X^{\rm ex}_\alpha(v)$. This further implies that $X^{\rm ex}_\alpha(u) < X^{\rm ex}_\alpha(r)$ for every $r \in (u,v)$, otherwise it would contradict either (H1) or (H4), depending on whether $X^{\rm ex}_\alpha$ admits a local minimum at $u$ or not. We conclude that in this case, we have $u \simeq^{X^{\rm ex}_\alpha} v$.

Suppose now that $\Delta X^{\rm ex}_\alpha(u) > 0$; then, by (H2), for every $\varepsilon > 0$, there exists $r \in (u, u+\varepsilon)$ such that $X^{\rm ex}_\alpha(r) < X^{\rm ex}_\alpha(u)$. Consequently, we must have $X^{n_k}(n_k^{-1} j_{n_k}) \to X^{\rm ex}_\alpha(u-)$ as $k \to \infty$, otherwise \eqref{eq:pas_d_idee_1} would give $X^{\rm ex}_\alpha(u) = X^{\rm ex}_\alpha(v) = X^{\rm ex}_\alpha(v-)$ and we would get a contradiction with \eqref{eq:pas_d_idee_2}. We thus have $X^{\rm ex}_\alpha(u-) = X^{\rm ex}_\alpha(v) \le X^{\rm ex}_\alpha(r)$ for every $r \in (u, v)$; moreover the latter inequality is strict since an element $r \in (u, v)$ such that $X^{\rm ex}_\alpha(r) = X^{\rm ex}_\alpha(u-)$ is the time of a local minimum of $X^{\rm ex}_\alpha$ and this contradicts (H4). We see again that $u \simeq^{X^{\rm ex}_\alpha} v$.

In the second case when each pair $(j_{n_k}, l_{n_k})$ satisfies $j_{n_k} = p(j_{n_k})+1$ and $l_{n_k} = \inf\{m \ge j_{n_k} : \W^n_m = \W^n_{p(j_{n_k})}\}$, the very same arguments apply, which completes the proof.
\end{proof}

\section{Extensions}
\label{sec:ext}

If $P_{n}$ is a simply generated non-crossing partition generated using a sequence of weights $w$, a natural question is to ask how behaves the largest block area of $P_{n}$. In this direction, if $P$ is a non-crossing partition, we propose to study ${P}^{\graybullet}$, which is by definition the union of the convex hulls of the blocks of $P$ (see Fig.~\ref{ex:hull} for an example).

 \begin{figure}[!h]
 \begin{center}
    \includegraphics[width=0.22 \linewidth]{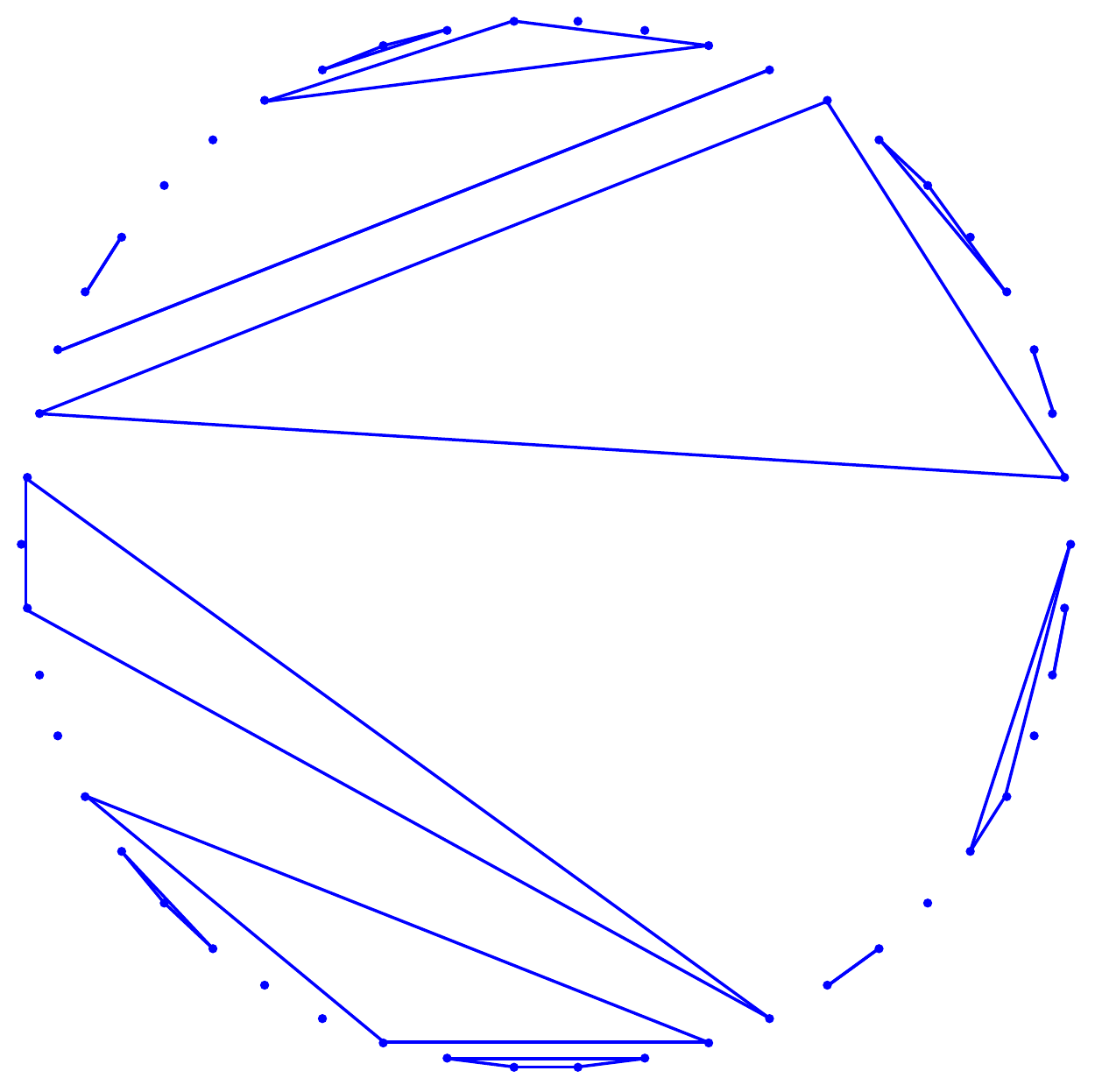} \hfill  
    \includegraphics[width=0.22 \linewidth]{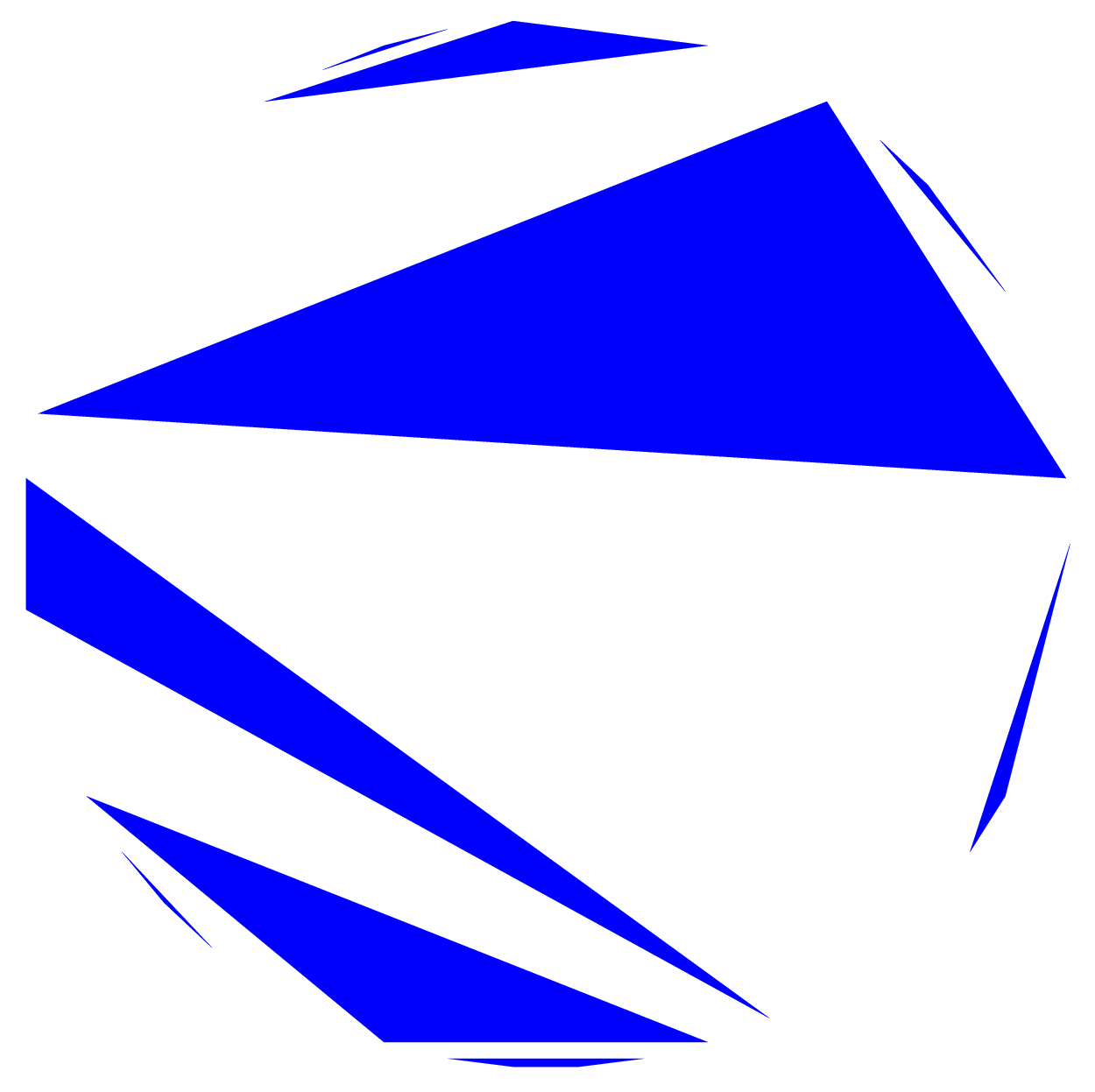} \hfill    \includegraphics[width=0.22 \linewidth]{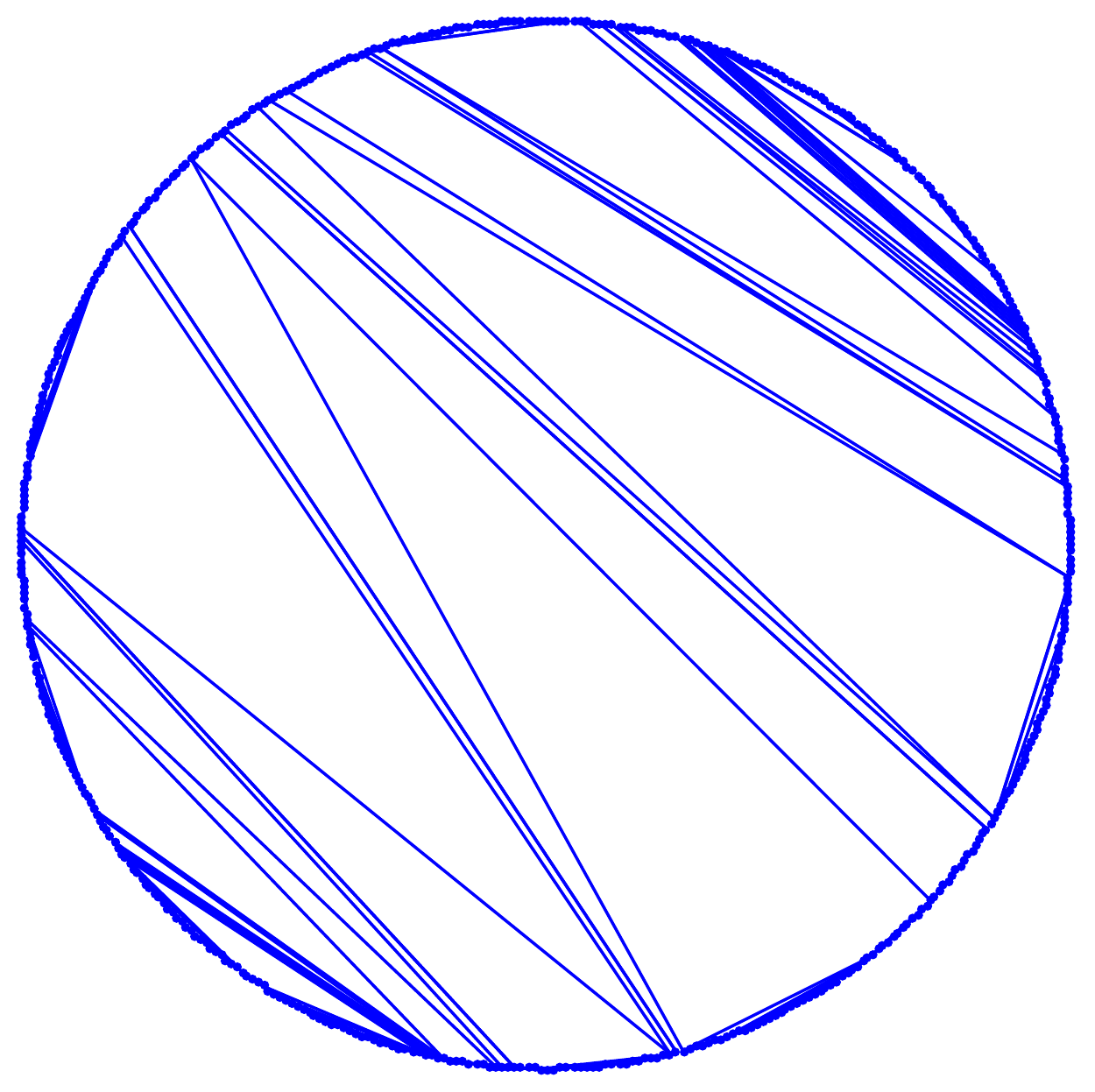}\hfill
    \includegraphics[width=0.22 \linewidth]{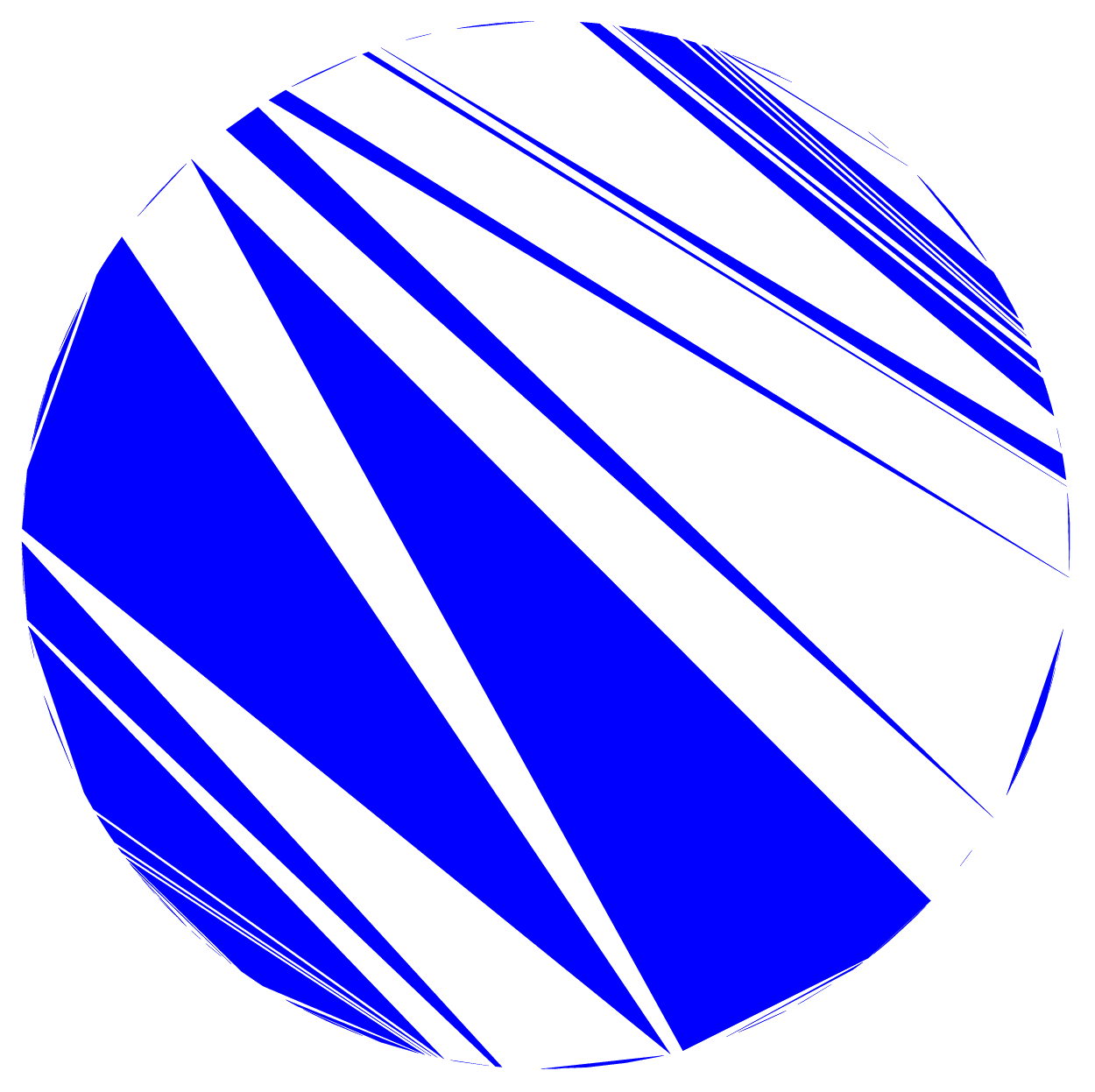}
   \caption{\label{ex:hull}From left to right: $P^{}_{50},P_{50}^{\graybullet},P^{}_{500},P^{\graybullet}_{500}$, where $P^{}_{50}$ (resp.~$P^{}_{500}$) is a uniform non-crossing partition of $[50]$ (resp.~$[500]$).}
 \end{center}
 \end{figure}
 
\begin{ques}Assume that the weights $w$ are equivalent to a critical probability distribution which has finite variance. Is it true that $P^{\graybullet}_{n}$ converges in distribution as $n \rightarrow \infty$ to a random compact subset of the unit disk?
\end{ques}
 
If the answer was positive, the limiting object would be obtained from the Brownian triangulation by ``filling-in'' some triangles, and this would imply that the largest block area of $P_{n}$ converges in distribution to the area of the largest ``filled-in face'' of the distributional limit. 

In the case of $ \mathcal{A}$-constrained uniform plane partitions,  numerical simulations based on the calculation of the total area of $P^{\graybullet}_{n}$ indicate that this limiting distribution should depend on the weights $\mathcal{A}$ (note that in the particular case $\mathcal{A} \subset \{1,2\}$ it is clear that $(P_{n},P^{\graybullet}_{n}) \rightarrow (\mathbf{L}_2, \mathbf{L}_2)$ in distribution as $n \rightarrow \infty$).

\begin{figure}[!h] 
 \begin{center}
   \includegraphics[width=.35 \linewidth]{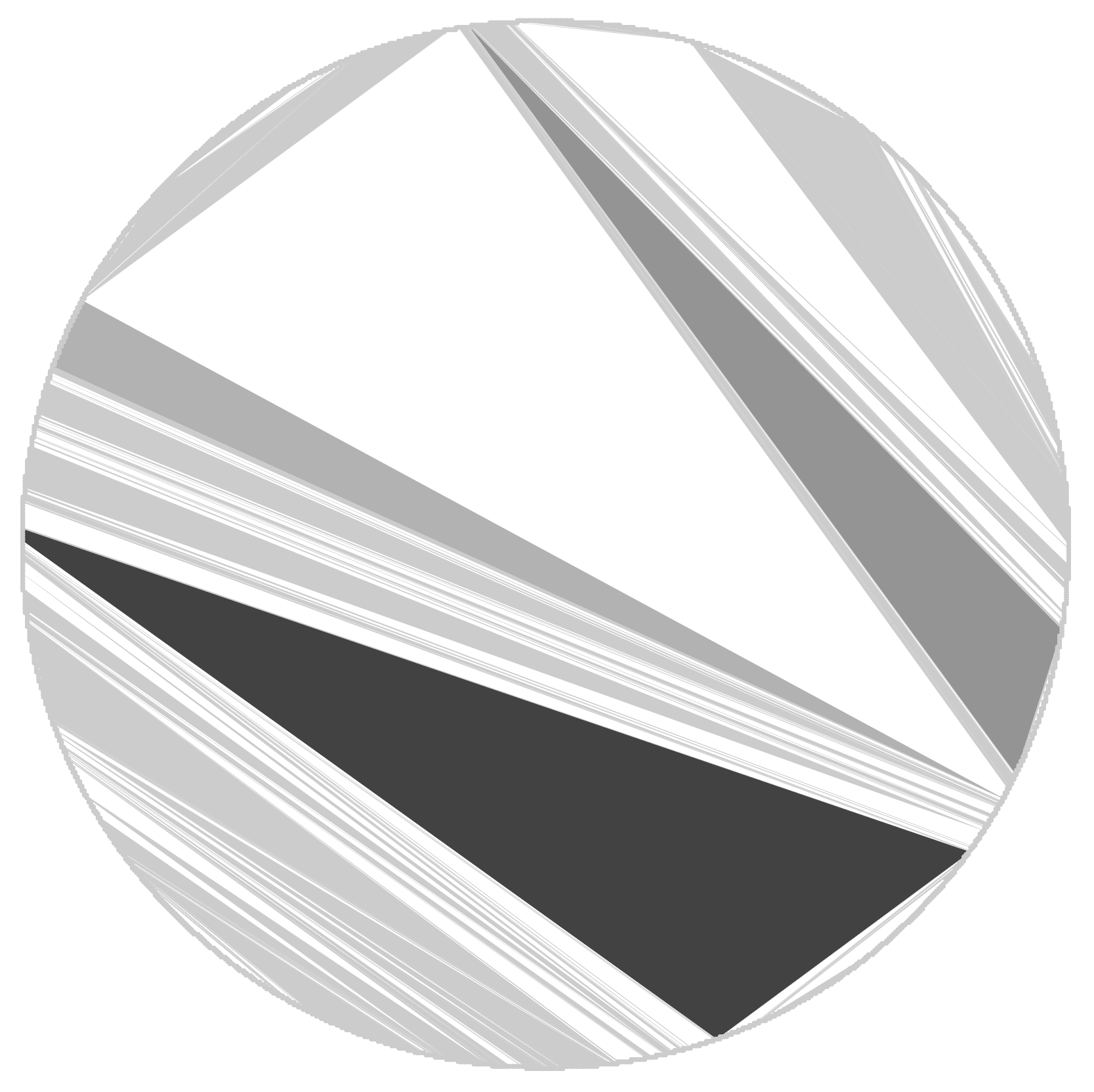} \hfill \includegraphics[width=0.35 \linewidth]{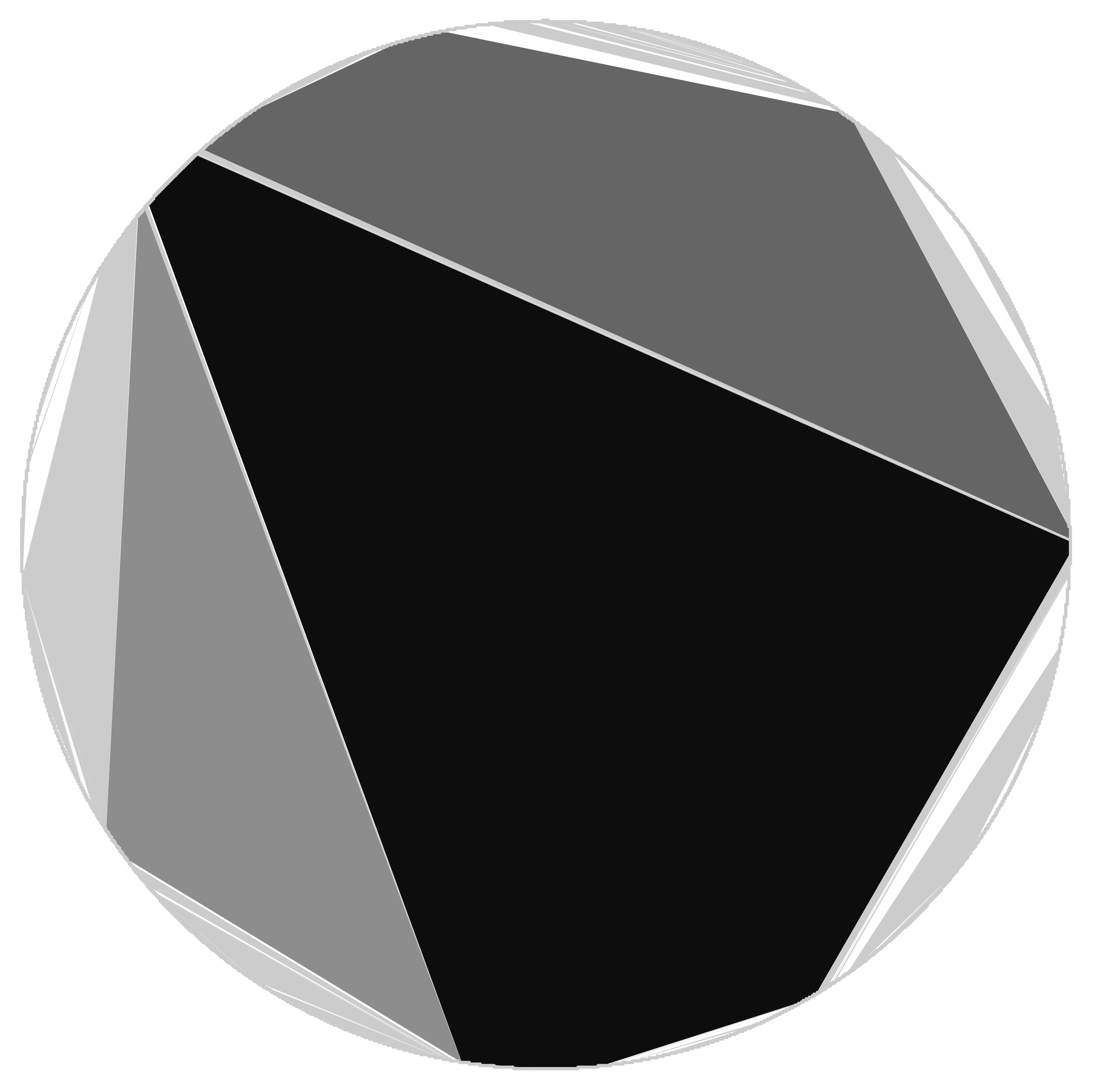}
    \caption{\label{ex:stable}A simulation of $P^{\graybullet}_{20000}$ for respectively $\alpha=2$ and $\alpha=1.3$, where the largest faces are the darkest ones.}
    \end{center}
    \end{figure}
    
When the weights $w$ are equivalent to a critical probability distribution that belongs to the domain of attraction of a stable law of index $\alpha \in(1,2)$, it is not difficult to adapt the arguments of the previous section to check that
$$ \left( P_{n}, \overline{\overline{\mathbb{D}} \setminus P^{\graybullet}_{n}} \right)   \quad \mathop{\longrightarrow}^{(d)}_{n \rightarrow \infty} \quad ( \mathbf{L}_{\alpha}, \mathbf{L}_{\alpha}),$$
meaning that the faces of $P_{n}$ cover in the limit the whole disk (see Fig.~\ref{ex:stable} for an illustration). In particular, in this case, the largest block area of $P_{n}$ converges in distribution to the largest area face of $ \mathbf{L}_{\alpha}$.

\bibliographystyle{siam}
\end{document}